\documentclass[11pt]{article}

\usepackage[T1]{fontenc}
\usepackage{lmodern}

\usepackage{amsfonts,amssymb,amsmath,amsthm}
\usepackage{mathtools}

\usepackage{graphicx}
\usepackage{float,wrapfig}
\usepackage[font=small, labelfont=bf]{caption}
\usepackage[font=small]{subcaption}
\usepackage{tikz}
\usetikzlibrary{shapes,shadows,arrows}

\usepackage{multirow}
\usepackage{tabularx}
\usepackage{array}
\usepackage{booktabs}

\usepackage{xcolor}
\usepackage{enumerate}
\usepackage{enumitem}
\usepackage{lineno}
\usepackage{comment}
\usepackage{commath}
\usepackage{calligra}
\usepackage{authblk}
\usepackage{fancyhdr}

\usepackage{cite}

\newtheorem{theorem}{Theorem}[section]
\newtheorem{lemma}{Lemma}[section]

\theoremstyle{definition}

\newtheorem{example}{Example}[section]

\theoremstyle{remark}
\newtheorem{remark}{Remark}[section]

\allowdisplaybreaks
\numberwithin{equation}{section}

\def\be{\begin{equation}}\def\ee{\end{equation}}
\def\ba{\begin{array}}\def\ea{\end{array}}
\def\bfg{\begin{figure}}\def\efg{\end{figure}}
\def\2pth{two point Hermite polynomial}
\def\c01{\mathcal{C}^{r}([0, 1])}

\def\01{[0, 1]}

\usepackage[colorlinks]{hyperref}
\usepackage[noabbrev]{cleveref}
\hypersetup{urlcolor=blue, colorlinks=true, citecolor=blue, linkcolor=black}

\crefname{theorem}{Theorem}{Theorems}
\crefname{lemma}{Lemma}{Lemmas}
\crefname{corollary}{Corollary}{Corollaries}
\crefname{proposition}{Proposition}{Propositions}
\crefname{conjecture}{Conjecture}{Conjectures}
\crefname{definition}{Definition}{Definitions}
\crefname{example}{Example}{Examples}
\crefname{remark}{Remark}{Remarks}

\Crefname{theorem}{Theorem}{Theorems}
\Crefname{lemma}{Lemma}{Lemmas}
\Crefname{corollary}{Corollary}{Corollaries}
\Crefname{proposition}{Proposition}{Propositions}
\Crefname{conjecture}{Conjecture}{Conjectures}
\Crefname{definition}{Definition}{Definitions}
\Crefname{example}{Example}{Examples}
\Crefname{remark}{Remark}{Remarks}

\author[$\dagger$]{Prakash Nainwal}
\author[$\dagger$]{Akash Anand}
\affil[$\dagger$]{%
Department of Mathematics and Statistics\\
Indian Institute of Technology Kanpur, India
}
 \date{}

\begin{document}

\title{A Generalized FC-Gram Approximation Framework with Analysis and Applications}
\maketitle

\begin{abstract}
The FC-Gram algorithm constructs high-order trigonometric approximations of nonperiodic functions by periodically extending them to a larger interval, with the quality of the blending continuation of Gram polynomials over the extension interval directly governing the approximation accuracy. We introduce GenFC, a generalized FC-Gram framework in which the continuation of each Gram polynomial is shaped by a cutoff function satisfying prescribed boundary flatness conditions. We establish a convergence theorem showing that for any such family the GenFC approximation error satisfies $O(n^{-\min(r+\beta,\,d)})$ in the supremum norm on the original interval, where $f \in C^r([0,1])$ has an integrable $(r+1)$th derivative, $d$ is the number of Gram polynomials, and $\beta \in [0,1]$ is the Fourier decay exponent of $f^{(r+1)}$. The modified FC-Gram algorithm, recently introduced by the authors, is recovered as a special case, and several explicit families satisfying these conditions are constructed in the paper. Numerical experiments across smooth, limited-regularity, and rapidly oscillating test cases confirm the theoretical predictions. The framework is further applied to high-order solvers for linear ODEs and parabolic PDEs via backward differentiation formulae (BDF) time-stepping, demonstrating high-order accuracy throughout.
\end{abstract}

% ********* INTRODUCTION******************* %  

 \vspace*{0.05in}

\section{Introduction}
In scientific computing, one frequently approximates a function known only at the points of a discrete grid, and the accuracy of this approximation bounds that of every computation built upon it. Polynomial interpolation on uniform grids is a natural choice for this purpose. However, the classical example of Runge~\cite{runge1901} shows that this approach can fail even for analytic functions, with the interpolating polynomial developing large oscillations near the endpoints as the number of nodes grows (see~\cite{platte2011impossibility} for a detailed discussion of approximating functions from equispaced data). Moving to non-uniform grids, such as Chebyshev nodes, resolves this for polynomial interpolants and is a widely adopted remedy.

In this paper, we study function approximation by trigonometric interpolation, a powerful alternative to interpolation by algebraic polynomials. On a uniform grid, trigonometric interpolants achieve spectral accuracy and can be evaluated in $O(n \log n)$ operations, making them especially attractive in practice. For periodic functions this works without issue. However, when the function is not periodic, its naive periodic extension has a jump discontinuity at the boundary of the computational interval, and the Fourier interpolant responds with spurious oscillations regardless of any grid refinement. This is the Gibbs phenomenon~\cite{gottlieb1997gibbs, hewitt1979gibbs}, which has motivated a range of remedies. Filtering the Fourier coefficients~\cite{tadmor2007filters} suppresses the oscillations in the interior of the interval but does not genuinely restore accuracy near the boundary. Reconstruction via Gegenbauer polynomials~\cite{gelb2006robust, gottlieb1997gibbs, gottlieb1992gibbs} can recover exponential accuracy from the truncated Fourier series, but the associated linear systems become exponentially ill-conditioned and are difficult to solve stably. Pad\'{e} and rational approximation methods \cite{driscoll2001pade, geer1995rational} have also been studied in this setting.

A different line of attack is to periodically extend the function to a larger interval $[0,\,b]$, $b > 1$, and obtain the approximation on $[0,\,1]$ from the resulting periodic data via FFT. The Fourier interpolant of the extended function serves as a high-order approximation to $f$ on the original interval. Methods following this approach are known as Fourier continuation (FC) or Fourier extension methods, and they have been developed and applied by a number of researchers \cite{anand2019fourier, boyd2002comparison, bruno2007accurate, garbey2000some, gruberger2021two, huybrechs2010fourier, lyon2011fast, matthysen2016fast}. 
Among these, the FC-Gram algorithm has been employed in the construction of high-order solvers for PDEs on complex geometries~\cite{albin2011spectral, amlani2016fc, bruno2022, bruno2010high, brunopaul2022,fontana2022, gaggioli2022}.

The FC-Gram construction for $f$ defined on $[0,\,1]$ works in two steps. Function data near each endpoint is projected onto a basis of Gram polynomials, and each basis polynomial is then given a blending-to-zero continuation over a short extension interval $[1,\,b]$. In the original algorithm, these continuations are determined implicitly by a least-squares procedure \cite{amlani2016fc, bruno2010high}. A key design choice is to keep the number of extension points fixed independent of $n$, which allows the blending continuations of the Gram polynomials to be precomputed offline and reused for any $n$, and keeps the FFT size moderate, making the algorithm computationally attractive in practice. Numerical experiments in the literature consistently demonstrate rapid convergence of the resulting approximations, though a rigorous theoretical justification for this behavior remains an open problem.

A number of variations of this construction have since been proposed, including \cite{albin2014discrete, amlani2016fc, Nainwal_Anand_ModFC_2025}. In~\cite{Nainwal_Anand_ModFC_2025} we replaced the implicit least-squares continuations of Gram polynomials with explicit ones constructed from two-point Hermite polynomials, which made a rigorous convergence analysis possible. We call this variant the modified FC-Gram algorithm, or ModFC. The resulting continuation is, however, uniquely determined by the derivative values of the Gram polynomials at the endpoints, leaving no flexibility in shaping the blending continuation over $[1,\,b]$. In practice this matters because the approximation error depends not only on the asymptotic convergence rate but also on the supremum norm of the continuation over $[1,\,b]$. A large supremum norm inflates the constant in the error bound and can cause the method to stagnate at an error level well above machine precision before the asymptotic regime is reached. Functions with steep gradients or rapid oscillations near the boundary tend to produce large Gram projection coefficients, and the ModFC continuation inherits this magnitude. Two such cases are demonstrated in the numerical experiments.

The present paper develops a generalized FC-Gram framework, GenFC, in 
which the blending continuation of each Gram polynomial is constructed by multiplying it by a shape function carrying a free parameter that governs the behavior of the continuation over $[1,\,b]$. The choice of this parameter provides explicit control over the supremum norm of the blending continuation over $[1,\,b]$, and an appropriate choice leads to improved approximation accuracy. The main result of the paper, Theorem~\ref{thm:main}, shows that for $f \in C^r([0,1])$ with integrable $(r+1)$th derivative satisfying a Fourier decay condition with exponent $\beta \in [0,1]$, the GenFC approximation error satisfies
\[
 \bigl\|\tau_n^b(e[p]f) - f\bigr\|_{\infty,[0,\,1]}
 = O\left(n^{-\min(r+\beta,\,d)}\right),
\]
where $d$ is the number of Gram polynomials, and holds for any admissible shape-function family. In particular, ModFC of~\cite{Nainwal_Anand_ModFC_2025} is recovered as the special case in which the shape functions are constructed from two-point Hermite polynomial continuations, and three additional families satisfying the admissibility conditions are constructed in the paper. Numerical experiments on smooth functions, functions with limited regularity, and functions with rapid oscillations near the boundary confirm the predicted convergence rates. The proposed approximation scheme is further used to build high-order solvers for linear ODEs and parabolic PDEs in one dimension, demonstrating high-order accuracy in both settings.

The paper is organized as follows. Section~\ref{sec:fc_extension} introduces the GenFC framework. Section~\ref{sec:errorAnalysis} develops the convergence analysis. Section~\ref{sec:ShapeFunc} 
constructs several admissible shape-function families. Section~\ref{sec:numerics_GenFC} presents numerical experiments confirming the theoretical predictions. Section~\ref{sec:GenFC_application} demonstrates the applicability of the proposed algorithm to high-order solvers for linear ODEs and parabolic PDEs. Section~\ref{sec:conclusion} offers concluding remarks.

\section{Generalized FC-Gram Framework}\label{sec:fc_extension} 
Let $f$ be a smooth function defined on $[0,\,1]$, and consider the equispaced grid on $[0,\,1]$ with spacing $h =1/n$, where $n\in\mathbb{N}$. The objective is to construct a trigonometric polynomial that interpolates $f$ on this grid. The FC-Gram algorithm computes this interpolant in the following two steps:

\subsection{Extension}
The first step is to construct a periodic extension of $f$ on a larger interval $[0,\,b]$, $b>1$, defined by
\begin{align}\label{eq:fc}
f^e(x)=
\begin{cases}
f(x), & x\in[0,\,1],\\
p(x), & x\in(1,b).
\end{cases}
\end{align}
Since several choices are possible for the continuation function $p$, the FC-Gram algorithm 
\cite{amlani2016fc,lyon2010high,Nainwal_Anand_ModFC_2025} constructs $p$ in two steps:
(i) projecting a small portion of $f$ near the boundary onto the space of Gram polynomials, and 
(ii) constructing a blending-to-zero continuation of Gram polynomials.

\subsubsection{Gram polynomials}

For $d\in\mathbb{N}$, let $d$ equispaced nodes on $[-1,1]$ be
\[
y_j=-1+\frac{2j}{d-1},\qquad j=0,\dots,d-1,
\]
and define the discrete inner product
\[
\langle p, q \rangle=\sum_{j=0}^{d-1}p(y_j)q(y_j).
\]
The Gram polynomials $\{p_\ell\}_{\ell=0}^{d-1}$ satisfy $ \langle p_k, p_\ell\rangle =\delta_{k\ell},\, p_\ell\in\mathcal{P}_\ell$, where $\mathcal{P}_\ell$ denotes the space of polynomials of degree at most $\ell$. The constant polynomial with $\langle p_0,p_0\rangle=1$ is $p_0(y)=1/\sqrt{d}$. The linear polynomial satisfying $\langle p_1,p_0\rangle=0$ and $\langle p_1,p_1\rangle=1$ is
\[
p_1(y)=\sqrt{\frac{3(d-1)}{d(d+1)}}\,y,\qquad d\ge2,
\]
and the quadratic polynomial satisfying $\langle p_2,p_0\rangle=0$, $\langle p_2,p_1\rangle=0$, and $\langle p_2,p_2\rangle=1$ is
\[
p_2(y)=\sqrt{\frac{5(d-1)^3}{4(d+1)(d^2-4)}}\left(3y^2-\frac{d+1}{d-1}\right),\qquad d\ge3.
\]

Higher-degree Gram polynomials may be obtained via the Gram--Schmidt orthogonalization
\begin{align*}
q_\ell(y)&=y^\ell-\sum_{k=0}^{\ell-1}\langle y^\ell,p_k\rangle p_k(y),\qquad
p_\ell(y)=\frac{q_\ell(y)}{\sqrt{\langle q_\ell,q_\ell\rangle}}.
\end{align*}
They also satisfy a three-term recurrence relation \cite{gautschi2004orthogonal, hildebrand1987introduction}. 
Since the inner product $\langle\cdot,\cdot\rangle$ only requires the values $p_\ell(y_j)$, $j=0,\dots,d-1$, these can be computed efficiently using recurrence relations or the QR-based procedure described in \cite{amlani2016fc}.

\subsubsection{Projection onto Gram Polynomials}
Once we obtain the set of Gram polynomials, we project small boundary data coming from $f$ onto the space of these polynomials. Given $n \in \mathbb{N}$, 
let $\delta := (d-1)/n$. We introduce affine maps
\begin{equation}
\varphi_L(x) = \frac{2(x-b)}{\delta} - 1, \quad 
\varphi_R(x) = \frac{2(x-1)}{\delta} + 1,
\end{equation}
where $\varphi_L$ maps $[b,\, b+\delta]$ onto $[-1,\,1]$ and $\varphi_R$ maps $[1-\delta,\, 1]$ onto $[-1,\,1]$. Note that under the $b$-periodic identification, the interval $[b,\, b+\delta]$ corresponds to the left neighborhood $[0, \delta]$ of the original domain, so $\varphi_L$ effectively maps the left boundary data of $f$ onto $[-1, 1]$. The \emph{left} and \emph{right Gram polynomials} are defined as
\begin{equation}
p_\ell^L(x) := (p_\ell \circ \varphi_L)(x), \quad 
p_\ell^R(x) := (p_\ell \circ \varphi_R)(x), \quad 
\ell = 0, \ldots, d-1.
\end{equation}
These are polynomials of degree $\ell$ that extend to all of $\mathbb{R}$; the intervals $[b,b+\delta]$ and $[1-\delta,1]$ serve as reference intervals for spatial scaling, but $p_\ell^L$ and $p_\ell^R$ will be evaluated on the extension interval $[1,\,b]$ when constructing the continuation in Section~\ref{sec:GramPolyCont}.

The projection operators are defined as:
\begin{equation}\label{eq:projection_operators}
(P_d^L f)(x) = \sum_{\ell=0}^{d-1} a_\ell^L p_\ell^L(x), \quad 
(P_d^R f)(x) = \sum_{\ell=0}^{d-1} a_\ell^R p_\ell^R(x),
\end{equation}
where $a_\ell^L := \langle f(({\cdot}+1)\delta/2), p_\ell \rangle$ and 
$a_\ell^R := \langle f(1 + ({\cdot}-1)\delta/2), p_\ell \rangle$ are obtained by sampling $f$ on $[0,\delta]$ and $[1-\delta, 1]$ respectively.

\subsubsection{Continuation of Gram Polynomials}\label{sec:GramPolyCont}
In the FC-Gram framework, in order to obtain $b$-periodic continuation $f^e$, the required extension $p$ on $[1, b]$ is constructed utilizing the boundary projections coefficient $a_\ell^L$ and $a_\ell^R$ along with blending-to-zero continuation of the Gram polynomials defined on $[1, b]$. More precisely, the FC-Gram algorithm constructs $p$ as:
\begin{align}\label{p}
p(x)=\sum_{\ell=0}^{d-1} a_\ell^R p_{\ell}^{R,e}(x)+ \sum_{\ell=0}^{d-1} a_\ell^L p_{\ell}^{L,e}(x),\quad x\in [1, b]
\end{align}
where $p_\ell^{R, e}$ (resp. $p_\ell^{L, e}$) is constructed such that it blends to zero while approaching $x=b$ (resp. $x=1$). As there could be several ways to construct these blending continuations~\cite{amlani2016fc, lyon2011fast, Nainwal_Anand_ModFC_2025}, in this paper we introduce a strategy that utilizes the fact that $p_\ell^R$ (resp. $p_\ell^L$), being a polynomial, is known on $[1,b]$. Multiplying this polynomial by a smooth cutoff function $\eta_\ell^R$ (resp.\ $\eta_\ell^L$) that vanishes as $x$ approaches $b$ (resp.\ 1) then yields an extension $p_\ell^{R,e}$ (resp.\ $p_\ell^{L,e}$) that also vanishes smoothly as the argument approaches $b$ (resp.\ 1). In other words, we define
\begin{align}\label{eq:blendingCont}
 p_\ell^{R,e}(x) := p_\ell^R(x) \eta_\ell^R(x), \quad p_\ell^{L,e}(x) := p_\ell^L(x) \eta_\ell^L(x), \quad 0\leq \ell\leq d-1,
\end{align}
for $x\in [1, b]$. Moreover, for given $d \in \mathbb{N}$, we require $\eta_\ell^{L,R} \in C^{d-1}([1,b])$ that is piecewise $C^{d+2}$, with $(\eta_\ell^{L,R})^{(d+2)} \in L^1([1,b])$. In addition, for given $\ell = 0, \ldots, d-1$, we require $\eta_\ell^{L,R}$
to satisfy, for $m = 0, \ldots, d-1$,
\begin{equation}\label{assump:eta}
 (\eta_\ell^R)^{(m)}(1) = \delta_{m0}, \quad (\eta_\ell^R)^{(m)}(b) = 0,
 \qquad
 (\eta_\ell^L)^{(m)}(b) = \delta_{m0}, \quad (\eta_\ell^L)^{(m)}(1) = 0.
\end{equation}

 The collection $(\eta_\ell^{L,R})_{\ell=0}^{d-1}$, will be referred to as a \emph{family of shape functions}, since they determine the shape of the blending continuation of Gram polynomials. By construction, these blending continuations inherit the regularity properties of the cutoff functions. Specifically, for $0\le k\le d-1$, we have
\begin{equation}\label{eq:PropertiesOfBlendingCont}
\begin{aligned}
 (p_\ell^{R,e})^{(k)}(1) &= (2/\delta)^k p_\ell^{(k)}(1),\quad (p_\ell^{R,e})^{(k)}(b) = 0,\\
 (p_\ell^{L,e})^{(k)}(b) &= (2/\delta)^k p_\ell^{(k)}(-1),\quad (p_\ell^{L,e})^{(k)}(1) = 0.
\end{aligned}
\end{equation}
For simplicity of notation, we define
\begin{align*}
f^L(y)&:=f((y+1)\delta/2), \qquad
f^R(y):=f(1+(y-1)\delta/2),
\end{align*}
whenever $y\in[-1,1]$.

\subsection{Approximation}
Write $b = 1 + (c+1)h$ with $c := (b-1)n - 1$, so that $c$ equispaced values outside $[0,\,1]$ are required to render $f^e$ periodic; since $b$ is kept fixed, $c = (b-1)n - 1 \sim n$. Once the extension $e[p](f)$ is constructed, we obtain periodic discrete data on an extended grid. Let $\{x_j\}_{j=0}^{n+c}$ denote the equispaced grid on $[0,b)$ given by
\[
x_j=\frac{jb}{n+c+1}, \qquad j=0,\dots,n+c .
\]
The data associated with the periodic extension are
\begin{equation}
(e[p](f))_j =
\begin{cases}
f_j, & j=0,\dots,n,\\
p_j, & j=n+1,\dots,n+c.
\end{cases}
\end{equation}
The trigonometric polynomial that interpolates $e[p](f)$ on this extended grid is
\begin{align}\label{eq:tnc}
\left(t_{n,c}(e[p]f)\right)(x)
=
\sum_{j=0}^{n+c} (e[p]f)_j L_j^{(n+c)}(x),
\qquad 0\le x\le b .
\end{align}
The Lagrange basis, $L_j^{(n+c)}$ are given by
\begin{align*}
L_j^{(n+c)}(x)=
\begin{cases}
1, & x=x_j,\\[6pt]
\dfrac{1}{n+c+1}
\left(
\sin(\pi(j-nx))
\cot\left(\dfrac{\pi(j-nx)}{n+c+1}\right)
\right), & x\ne x_j,
\end{cases}
\end{align*}
if $n+c$ is odd, and
\begin{align*}
L_j^{(n+c)}(x)=
\begin{cases}
1, & x=x_j,\\[6pt]
\dfrac{1}{n+c+1}
\left(
\sin(\pi(j-nx))
\csc\left(\dfrac{\pi(j-nx)}{n+c+1}\right)
\right), & x\ne x_j.
\end{cases}
\end{align*}
if $n+c$ is even. The derivation of these expressions follows the procedure described in Section~11.3 of \cite{Linear_Integral_Eq_2014}.

\section{Convergence Analysis} \label{sec:errorAnalysis}
In this section we analyze the convergence properties of the generalized FC-Gram approximation method. Given $f$ on $[0,\,1]$, having constructed the approximation $t_{n,c}(e[p](f))$, we consider the approximation error
\[
\|f - t_{n,c}(e[p](f))\|_{\infty,[0,\,1]}.
\]
To analyze this error, we first introduce a suitable reference extension. Within this framework, we define
\begin{align}\label{p_ref_gen}
\left(e\left[p^{r,d,n}_{\eta,\text{ref}}\right]f\right)(x)
:=
\begin{cases}
f(x), & x\in [0,\,1],\\
p^{r,d,n}_{\eta,\text{ref}}(x), & x\in (1,b),
\end{cases}
\end{align}
where
\begin{equation}\label{pref}
\begin{aligned}
p^{r,d,n}_{\eta,\text{ref}}(x)
&=
\sum_{\ell = 0}^{\min(r,d-1)}
\left\langle P^T_{\min(r,d-1)}(f^L), p_{\ell}\right\rangle
p_{\ell}^{L,e}(x) \\
&\quad +
\sum_{\ell = 0}^{\min(r,d-1)}
\left\langle P^T_{\min(r,d-1)}(f^R), p_{\ell}\right\rangle
p_{\ell}^{R,e}(x),
\end{aligned}
\end{equation}
and $P^{T}_{k}(g)$ denotes the degree-$k$ Taylor polynomial of $g$. For notational convenience, we denote the continuation $p^{r,d,n}_{\eta,\text{ref}}$ and the associated extension defined in \eqref{p_ref_gen} by $p_{\text{ref}}$ and $e[p_{\text{ref}}]f$, respectively.

Next, we present a collection of lemmas that describe the relationship between the derivatives of $f$, the continuation $p$, and the reference continuation $p_{\text{ref}}$. Since we assume $f \in C^{r}([0,\,1])$, while the constructions of $p$ and $p_{\text{ref}}$ may involve an arbitrary number of Gram polynomials, these results make the resulting derivative relationships explicit.
\begin{lemma}\label{lemma:pref}
Let $r, \,d\in\mathbb{N}$, $f\in C^{r}([0,\,1])$, and $m=\min(r,d-1)$. For every
$k\ge 0$,
\[
 p_{\text{ref}}^{(k)}(1^+)
 =\begin{cases} f^{(k)}(1),&0\le k\le m,\\0,&m<k\leq d-1,\end{cases}
 \qquad
 p_{\text{ref}}^{(k)}(b^-)
 =\begin{cases} f^{(k)}(0),&0\le k\le m,\\0,&m<k\leq d-1.\end{cases}
\]
\end{lemma}
\begin{proof}
We prove the identities at $1^+$; those at $b^-$ follow by the same argument using $f^L$ and $p_\ell^{L,e}$ in place of $f^R$ and $p_\ell^{R,e}$. Set $q_R := P^T_m(f^R)$, the degree-$m$ Taylor polynomial of
$f^R(y) = f\left(1 + (y-1)\delta/2\right)$ at $y = 1$. Since $q_R \in \mathcal{P}_m$, it admits the exact Gram expansion
\[
 q_R(y) = \sum_{\ell=0}^{m} \langle q_R, p_\ell \rangle\, p_\ell(y),
\]
and from the definition of $p_{\text{ref}}$ we may write
\[
 p_{\text{ref}}(x)
 = \sum_{\ell=0}^{m} \langle q_R, p_\ell \rangle\, p_\ell^{R,e}(x)
 + \sum_{\ell=0}^{m} \langle q_L, p_\ell \rangle\, p_\ell^{L,e}(x).
\]
Since $m \le d-1$, the boundary conditions \eqref{eq:PropertiesOfBlendingCont} give
\begin{equation*}
 (p_\ell^{L,e})^{(k)}(1^+) = 0 \quad \text{and} \quad
(p_\ell^{R,e})^{(k)}(1^+) = (2/\delta)^k p_\ell^{(k)}(1) 
\end{equation*}

for all $k \le m$, so that
\begin{equation}\label{eq:pref_at_1}
 p_{\text{ref}}^{(k)}(1^+)
 = \left(\frac{2}{\delta}\right)^{k}\sum_{\ell=0}^{m} \langle q_R, p_\ell \rangle\, p_\ell^{(k)}(1)
 = \left(\frac{2}{\delta}\right)^{k} q_R^{(k)}(1),
 \qquad 0 \le k \le m,
\end{equation}
where the second equality uses the Gram expansion of $q_R$. The Taylor polynomial identity $q_R^{(k)}(1) = (f^R)^{(k)}(1)$ and the chain rule $(f^R)^{(k)}(1) = \left(\delta/2\right)^{k} f^{(k)}(1)$ combine with \eqref{eq:pref_at_1} to give $p_{\text{ref}}^{(k)}(1^+) = f^{(k)}(1)$ for $0 \le k \le m$. For $k > m$, every polynomial $p_\ell$ in the sum has $\deg p_\ell = \ell \le m < k$, so $p_\ell^{(k)} \equiv 0$ and hence
$p_{\text{ref}}^{(k)}(1^+) = 0$ for $m < k \le d-1$, completing the proof.
\end{proof}
The next lemma describes the relationship between the derivatives of $p$ and those of $f$.
\begin{lemma}\label{lemma:p}
Let $r, d \in \mathbb{N}$, $f \in C^{r}([0,\,1])$, and
$m = \min(r,\,d-1)$. For every $k$ with $0 \le k \le m$,
\begin{align*}
 p^{(k)}(1^+)
 &= f^{(k)}(1)
 + \left(\frac{2}{\delta}\right)^{k}
 \sum_{\ell=k}^{d-1}
 \langle R_m(f^R),\, p_\ell \rangle\, p_\ell^{(k)}(1),
 \\[6pt]
 p^{(k)}(b^-)
 &= f^{(k)}(0)
 + \left(\frac{2}{\delta}\right)^{k}
 \sum_{\ell=k}^{d-1}
 \langle R_m(f^L),\, p_\ell \rangle\, p_\ell^{(k)}(-1),
\end{align*}
where $R_m(f^L)$ and $R_m(f^R)$ are the degree-$m$ Taylor remainders of $f^L$ and $f^R$ at $y=-1$ and $y=1$ respectively.
\end{lemma}

\begin{proof}
 We derive the formula at $b^-$; the identity at $1^+$ follows by the same argument with $f^R$ and $p_\ell^{R,e}$ in place of
$f^L$ and $p_\ell^{L,e}$.

By construction, $(p_\ell^{R,e})^{(k)}(b^-) = 0$
for all $\ell$, and $(p_\ell^{L,e})^{(k)}(b^-) = (2/\delta)^kp_\ell^{(k)}(-1)$ for $k \le d-1$.
Since $m \le d-1$, for every $0 \le k \le m$,
\begin{equation}\label{eq:p_bm}
 p^{(k)}(b^-)
 = \left(\frac{2}{\delta}\right)^{k} \sum_{\ell=0}^{d-1}
 \langle f^L,\,p_\ell \rangle\, p_\ell^{(k)}(-1).
\end{equation}
Substituting $f^L = P^T_m(f^L) + R_m(f^L)$ into \eqref{eq:p_bm}
and using exactness of the Gram expansion for $P^T_m(f^L)
\in \mathcal{P}_m$, the Taylor polynomial part contributes
\[
 \left(\frac{2}{\delta}\right)^{k}
 \left(P^T_m(f^L)\right)^{(k)}(-1)
 = \left(\frac{2}{\delta}\right)^{k}(f^L)^{(k)}(-1)
 = f^{(k)}(0),
\]
where we used $f^L(y) = f\left((y+1)\delta/2\right)$ so that the chain rule gives
$(f^L)^{(k)}(-1) = (\delta/2)^k f^{(k)}(0)$.
Since $\deg p_\ell = \ell < k$ implies $p_\ell^{(k)} \equiv 0$,
the terms $\ell = 0, \dots, k-1$ in the remainder part of
\eqref{eq:p_bm} vanish, and combining with the above yields the stated formula.
\end{proof}

\begin{lemma}\label{lemma:EstimateOnPkRk}
Let $k$ be a nonnegative integer and $f \in C^{k+1}([0,\,1])$.
For $0 \le \ell \le k$,
\begin{align*}
 \left|\langle P^T_k(f^L), p_\ell \rangle\right|
 &\le C_1^{k,\ell}\,\delta^\ell,
 \qquad
 \left|\langle P^T_k(f^R), p_\ell \rangle\right|
 \le C_2^{k,\ell}\,\delta^\ell,\\[4pt]
 \left|\langle R^T_k(f^L), p_\ell \rangle\right|
 &\le D_1^{k,\ell}\,\delta^{k+1},
 \qquad
 \left|\langle R^T_k(f^R), p_\ell \rangle\right|
 \le D_2^{k,\ell}\,\delta^{k+1},
\end{align*}
where
\begin{align*}
 C_1^{k,\ell}
 &= \sum_{m=0}^{k-\ell}
 \frac{|f^{(m+\ell)}(0)|\,
 \|(u+1)^{m+\ell}\|_{2,[-1,1]}}
 {2^{m+\ell}(m+\ell)!},
 &
 D_1^{k,\ell}
 &= \frac{2}{k!}\,
 \|f^{(k+1)}\|_{\infty,[0,\,1]}\,
 \|p_\ell\|_{\infty,[-1,1]},\\[4pt]
 C_2^{k,\ell}
 &= \sum_{m=0}^{k-\ell}
 \frac{|f^{(m+\ell)}(1)|\,
 \|(u-1)^{m+\ell}\|_{2,[-1,1]}}
 {2^{m+\ell}(m+\ell)!},
 &
 D_2^{k,\ell}
 &= \frac{2}{k!}\,
 \|f^{(k+1)}\|_{\infty,[0,\,1]}\,
 \|p_\ell\|_{\infty,[-1,1]}.
\end{align*}
In particular, $C_i^{k,\ell}$ and $D_i^{k,\ell}$ are independent
of $\delta$.
\end{lemma}

\begin{proof}
We prove the bounds for $f^L$; those for $f^R$ follow by the same
argument. The Taylor polynomial of $f^L(y) = f\left((y+1)\delta/2\right)$
about $y = -1$ gives
\[
 \langle P^T_k(f^L), p_\ell \rangle
 = \sum_{m=0}^{k}
 \frac{f^{(m)}(0)}{m!}
 \left(\frac{\delta}{2}\right)^{m}
 \langle (y+1)^m, p_\ell \rangle.
\]
Since $p_\ell \perp \mathcal{P}_{\ell-1}$, the terms $m < \ell$ vanish.
Factoring out $\delta^\ell$ and applying Cauchy--Schwarz to the remaining inner products gives $|\langle P^T_k(f^L), p_\ell \rangle| \le C_1^{k,\ell}\delta^\ell$,
where the continuous $L^2([-1,1])$-norm in $C_1^{k,\ell}$ arises from the
equivalence of the discrete and continuous norms on polynomials of bounded
degree, with a constant depending only on the fixed parameter $d$.

For the remainder, the integral representation
\[
 (R^T_kf^L)(y)
 = \frac{1}{k!}
 \int_0^{(y+1)\delta/2}
 \left((y+1)\delta/2 - t\right)^{k}
 f^{(k+1)}(t)\,dt
\]
and H\"{o}lder's inequality give
\[
 \left|\langle R^T_k(f^L), p_\ell \rangle\right|
 \le
 \|R^T_k(f^L)\|_{1,[-1,1]}\,\|p_\ell\|_{\infty,[-1,1]}.
\]
Applying Fubini's theorem to interchange the order of integration,
\[
 \|R^T_k(f^L)\|_{1,[-1,1]}
 \le \frac{1}{k!}
 \int_0^\delta |f^{(k+1)}(t)|
 \left(
 \int_{2t/\delta-1}^{1}
 \left((y+1)\delta/2 - t\right)^k dy
 \right) dt.
\]
The substitution $u = (y+1)\delta/2 - t$ converts the inner integral to
$\int_0^{\delta-t} u^k\,(\delta/2)^{-1}\,du \le 2\delta^k$,
and the stated bound
$|\langle R^T_k(f^L), p_\ell \rangle| \le D_1^{k,\ell}\delta^{k+1}$
follows.
\end{proof}

The next lemma will play a crucial role while proving the main theorem later. 

\begin{lemma}\label{lemma:sufficientCondEta}
Let $r, d \in \mathbb{N}$, $m = \min(r,d-1)$, $s = \min(r,d)$,
and let $p$ and $p_{\text{ref}}$ be constructed with $d$ Gram
polynomials and shape-function family $(\eta_\ell^L,\, \eta_\ell^R)_{\ell=0}^{d-1}$. Then
\begin{equation}\label{eq:Fourier_decay}
 \begin{aligned}
 \left|\int_1^b p_{\text{ref}}^{(s+1)}(x)\,
 e^{-2\pi i\ell x/b}\,dx\right| &= O(|\ell|^{-1}),
 \\
 \left|\int_1^b (p - p_{\text{ref}})^{(s+1)}(x)\,
 e^{-2\pi i\ell x/b}\,dx\right| &= O(|\ell|^{-1}).
\end{aligned} 
\end{equation}

\end{lemma}

\begin{proof}
A single integration by parts on $[1,\,b]$ gives
\[
 \left|\int_1^b g(x)\,e^{-2\pi i\ell x/b}\,dx\right|
 \le
 \frac{b}{2\pi|\ell|}
 \left(|g(b^-)| + |g(1^+)| + \|g'\|_{1,[1,b]}\right).
\]
Applying this with $g = p_{\text{ref}}^{(s+1)}$ and $g = (p-p_{\text{ref}})^{(s+1)}$, both claims
in~\eqref{eq:Fourier_decay} reduce to showing that the boundary values and the $L^1$ norm of the next derivative are bounded uniformly in $n$. The boundary values of $p_{\text{ref}}^{(s+1)}$
at $1^+$ and $b^-$ follow from \eqref{eq:PropertiesOfBlendingCont} and Lemma~\ref{lemma:pref}, and those of
$(p-p_{\text{ref}})^{(s+1)}$ follow from Lemmas~\ref{lemma:pref} and~\ref{lemma:p}, so it suffices to bound
$\|p_{\text{ref}}^{(s+2)}\|_{1,[1,b]}$ and $\|(p-p_{\text{ref}})^{(s+2)}\|_{1,[1,b]}$ uniformly in $n$.

We begin with $p_\ell^{R,e} = p_\ell^R\eta_\ell^R$. The Leibniz rule gives
\begin{equation}\label{eq:leibniz_expand}
 \left\|(p_\ell^{R,e})^{(s+2)}\right\|_{1,[1,b]}
 \le
 \sum_{j=0}^{s+2}\binom{s+2}{j}
 \left\|(\eta_\ell^R)^{(s+2-j)}\right\|_{1,[1,b]}
 \left\|(p_\ell^R)^{(j)}\right\|_{1,[1,b]}.
\end{equation}
Since $\deg p_\ell = \ell$, every term with $j > \ell$ vanishes.
For the remaining terms, the chain rule applied to
$p_\ell^R = p_\ell \circ \varphi_R$ gives
$(p_\ell^R)^{(j)}(x) = (2/\delta)^j\,p_\ell^{(j)}(\varphi_R(x))$.
The substitution $\xi = \varphi_R(x)$, with Jacobian $dx/d\xi = \delta/2$,
gives
\begin{equation}\label{eq:L1_cov}
 \left\|(p_\ell^R)^{(j)}\right\|_{1,[1,b]}
 =
 \left(\frac{2}{\delta}\right)^{j-1}
 \left\|p_\ell^{(j)}\right\|_{1,\varphi_R([1,b])}.
\end{equation}
Since $p_\ell^{(j)}$ is a polynomial of degree $\ell-j$ with
coefficients depending only on $d$, writing $\xi = 1+t$ for
$t \in [0,2(b-1)/\delta]$ and integrating directly,
\begin{equation}\label{eq:poly_int}
 \left\|p_\ell^{(j)}\right\|_{1,\varphi_R([1,b])}
 \le
 \frac{C_{\ell,j}}{\ell-j+1}
 \left(\frac{2(b-1)}{\delta}\right)^{\ell-j+1},
\end{equation}
where $C_{\ell,j}$ depends only on $d$ and the Gram polynomials,
independent of $n$. Substituting~\eqref{eq:poly_int}
into~\eqref{eq:L1_cov} and collecting powers of $\delta$,
\[
 \left\|(p_\ell^R)^{(j)}\right\|_{1,[1,b]}
 \le
 A_{\ell,j}\left(\frac{2}{\delta}\right)^{\ell},
 \qquad
 A_{\ell,j} :=
 \frac{C_{\ell,j}(2(b-1))^{\ell-j+1}}{(\ell-j+1)\,2^{\ell-j+1}},
\]
with $A_{\ell,j}$ independent of $n$. Inserting
into~\eqref{eq:leibniz_expand},
\begin{equation}\label{eq:single_piece_bound}
 \left\|(p_\ell^{R,e})^{(s+2)}\right\|_{1,[1,b]}
 \le
 \left(\frac{2}{\delta}\right)^{\ell}
 \sum_{j=0}^{\ell}\binom{s+2}{j}A_{\ell,j}
 \left\|(\eta_\ell^R)^{(s+2-j)}\right\|_{1,[1,b]},
\end{equation}
and the same bound holds for $p_\ell^{L,e}$ with constants $A_{\ell,j}^L$ of the same form.

To bound $\|p_{\text{ref}}^{(s+2)}\|_{1,[1,b]}$, recall from~\eqref{pref} that
\[
p_{\text{ref}} = \sum_{\ell=0}^{m}\langle P^T_m(f^L),p_\ell\rangle
p_\ell^{L,e} + \sum_{\ell=0}^{m}\langle P^T_m(f^R),p_\ell\rangle
p_\ell^{R,e}.
\]
The triangle inequality and~\eqref{eq:single_piece_bound}, utilizing the estimates obtained in Lemma~\ref{lemma:EstimateOnPkRk} yield the following
\begin{align*}
 \|p_{\text{ref}}^{(s+2)}\|_{1,[1,b]} &\leq \sum_{\ell=0}^{m}2^\ell C_1^{m, \ell}\left(\sum_{j=0}^{\ell}\binom{s+2}{j} A_{\ell,j}^L\left\|(\eta_\ell^L)^{(s+2-j)}\right\|_{1,[1,b]}\right)\\
 &+ \sum_{\ell=0}^{m}2^\ell C_2^{m, \ell}\left(\sum_{j=0}^{\ell}\binom{s+2}{j} A_{\ell,j}^R\left\|(\eta_\ell^R)^{(s+2-j)}\right\|_{1,[1,b]}\right)
\end{align*}
 Since $b$, $d$, the constants $C_i^{m,\ell}$, $A_{\ell,j}$, and the $L^1$ norms of the shape function derivatives over $[1,\,b]$ are all independent of $n$, the right-hand side is bounded uniformly in $n$.

The bound for $\|(p-p_{\mathrm{ref}})^{(s+2)}\|_{1,[1,b]}$ follows
by the same argument, with the Taylor polynomial coefficients
$\langle P^T_m(f^{L,R}),p_\ell\rangle$ replaced by the remainder
coefficients $\langle R_m(f^{L,R}),p_\ell\rangle$.
By Lemma~\ref{lemma:EstimateOnPkRk},
$|\langle R_m(f^R),p_\ell\rangle| \le D_2^{m,\ell}\delta^{m+1}$,
and analogously for $f^L$. Since $\ell \le m$, each coefficient
$D_i^{m,\ell}\delta^{m+1}$ contributes a factor $\delta^{m+1-\ell}\ge\delta>0$
after cancellation with $(2/\delta)^\ell$, which remains bounded
for all $n$. Hence the right-hand side is bounded uniformly in $n$, completing the proof.
\end{proof}

Having established the required auxiliary results, we turn to
estimating the approximation error. For $x \in [0, b]$, the error satisfies
\begin{align*}
 \left(e[p]f\right)(x) - \left(t_{n,c}(e[p]f)\right)(x)
 &= \left(e[p_{\text{ref}}]f\right)(x)
 - \left(t_{n,c}(e[p_{\text{ref}}]f)\right)(x)\\
 &\quad
 + \left(e[p]f - e[p_{\text{ref}}]f\right)(x)
 - \left(t_{n,c}(e[p]f - e[p_{\text{ref}}]f)\right)(x).
\end{align*}
Restricting to $x \in [0,\,1]$, where $(e[p]f)(x) = f(x)$, this decomposition gives
\begin{equation}\label{eq:error_decomp}
 f(x) - \left(t_{n,c}(e[p]f)\right)(x) = E_1(x) + E_2(x),
\end{equation}
where
\begin{equation}\label{eq:E1E2}
\begin{aligned}
 E_1(x) &:= \left(e[p_{\text{ref}}]f\right)(x)
 - \left(t_{n,c}(e[p_{\text{ref}}]f)\right)(x),\\
 E_2(x) &:= \left(e[p]f - e[p_{\text{ref}}]f\right)(x)
 - \left(t_{n,c}(e[p]f - e[p_{\text{ref}}]f)\right)(x).
\end{aligned}
\end{equation}
Each term has the form $g(x) - (\tau^b_n g)(x)$ for a suitable
$b$-periodic function $g$, where $\tau^b_n g$ denotes the trigonometric
polynomial of the form \eqref{eq:tnc} that interpolates $g$ on the equispaced
grid $\{x_j\}_{j=0}^{n+c}$. A bound on $\|\tau^b_n g - g\|_{\infty,[0,\,1]}$
in terms of the smoothness of $g$ and its Fourier-coefficient decay
therefore directly controls both $\|E_1\|_{\infty,[0,\,1]}$ and
$\|E_2\|_{\infty,[0,\,1]}$. The following lemma provides such a bound.

\begin{lemma}\label{GenFC:lemma_interp_error}
Let $b \in \mathbb{Q}$ with $b > 1$, let $s \in \mathbb{N} \cup \{0\}$,
and let $g$ be a $b$-periodic function of the form
\[
 g(x) =
 \begin{cases}
 g_1(x), & x \in [0,\,1],\\
 g_2(x), & x \in (1,b),
 \end{cases}
\]
where $g_1 \in C^s([0,\,1]) \cap C^{s+1}(0,1)$ and $g_2 \in C^s([1,b]) \cap C^{s+1}(1,b)$.
Suppose that for some $0 \le \gamma \le 1$ with $s + \gamma > 0$, and some $\ell_0 > 0$ and $M > 0$,
\begin{equation}\label{eq:Cond_On_g}
 \left|\int_0^1 g_1^{(s+1)}(x)\,e^{-2\pi i\ell x/b}\,dx\right|
 + \left|\int_1^b g_2^{(s+1)}(x)\,e^{-2\pi i\ell x/b}\,dx\right|
 \le M|\ell|^{-\gamma},
 \quad |\ell| > \ell_0.
\end{equation}
Then for all $n \ge 2\ell_0/b$,
\begin{equation}\label{eq:interp_bound}
 \|\tau^b_n g - g\|_{\infty,[0,\,1]}
 \le \sum_{k=1}^{s}
 \frac{4J_k\,\zeta(k+1)}{\pi^{k+1}}\,n^{-k}
 + \frac{4M\,\zeta(s+1+\gamma)}{\pi^{s+1+\gamma}}\,n^{-(s+\gamma)},
\end{equation}
where $\zeta(\,\cdot\,)$ denotes the Riemann zeta function, and
\begin{equation}\label{eq:J_k}
 J_k = \left|g_1^{(k)}(0^+) - g_2^{(k)}(b^-)\right|
 + \left|g_2^{(k)}(1^+) - g_1^{(k)}(1^-)\right|,
 \quad 1 \le k \le s.
\end{equation}
\end{lemma}

\begin{proof}
 Given an integer $\ell$, let
\begin{align}\label{eq:c_ell_g}
 c_{\ell}(g):= \dfrac{1}{b}\int_0^b g(x)e^{-2\pi i \ell x/b} dx, 
\end{align}
denote the Fourier coefficient of $g$. Then, for $\ell \ne 0$,
integrating by parts $s$ times on each of
$[0,\,1]$ and $[1,\,b]$ separately and collecting the boundary
contributions at $x = 0^+$, $x = 1^\pm$, and $x = b^-$ gives
\begin{align}\label{eq:cl_g}
 c_\ell(g)
 &= \frac{1}{b}\sum_{k=1}^{s}
 \left(\frac{b}{2\pi i\ell}\right)^{k+1}
 \left[
 \left(g_1^{(k)}(0^+) - g_2^{(k)}(b^-)\right)
 + \left(g_2^{(k)}(1^+) - g_1^{(k)}(1^-)\right)
 e^{-2\pi i\ell/b}
 \right]
 \notag\\
 &\quad
 + \frac{1}{b}
 \left(\frac{b}{2\pi i\ell}\right)^{s+1}
 \left(
 \int_0^1 g_1^{(s+1)}(x)\,e^{-2\pi i\ell x/b}\,dx
 + \int_1^b g_2^{(s+1)}(x)\,e^{-2\pi i\ell x/b}\,dx
 \right).
\end{align}
The aliasing formula for trigonometric interpolation on $n+c+1$ equispaced nodes on $[0,b)$ gives
$g(x) - (\tau^b_n g)(x)
= \sum_{\ell}\sum_{m \ne 0} c_{\ell+nbm}(g)\,e^{2\pi i\ell x/b}$, so
\[
 |g(x) - (\tau^b_n g)(x)|
 \le 2\sum_{\ell=-nb/2}^{(nb/2)-1}
 \sum_{m=1}^{\infty}
 |c_{\ell+nbm}(g)|.
\]
Inserting \eqref{eq:cl_g} and applying hypothesis~\eqref{eq:Cond_On_g} to bound the integral terms,
\begin{align}\label{eq:aliasing_bound}
 &|g(x) - (\tau^b_n g)(x)|
 \le \frac{2}{b}\sum_{k=1}^{s}
 \left(\frac{b}{2\pi}\right)^{k+1} J_k
 \sum_{\ell=-nb/2}^{(nb/2)-1}
 \sum_{m=1}^{\infty}
 \left(
 \frac{1}{|\ell+nbm|^{k+1}}
 + \frac{1}{|\ell-nbm|^{k+1}}
 \right)
 \notag\\
 &\quad \quad\quad
 + \frac{2M}{b}
 \left(\frac{b}{2\pi}\right)^{s+1+\gamma}
 \sum_{\ell=-nb/2}^{(nb/2)-1}
 \sum_{m=1}^{\infty}
 \left(
 \frac{1}{|\ell+nbm|^{s+1+\gamma}}
 + \frac{1}{|\ell-nbm|^{s+1+\gamma}}
 \right).
\end{align}
Since $|\ell| \le nb/2$ and $m \ge 1$, one has $|\ell \pm nbm| \ge nbm - nb/2 = nb(2m-1)/2$,
so $|\ell \pm nbm|^{-1} \le 2/(nb(2m-1))$. This bound is uniform in $\ell$, so after summing over the $nb$ admissible values of $\ell$, \eqref{eq:aliasing_bound} reduces to
\begin{align*}
 &|g(x) - (\tau^b_n g)(x)|
 \le \\
 &\quad \sum_{k=1}^{s}
 \frac{4J_k}{\pi^{k+1}}\,n^{-k}
 \sum_{m=1}^{\infty}\frac{1}{(2m-1)^{k+1}}
 + \frac{4M}{\pi^{s+1+\gamma}}\,n^{-(s+\gamma)}
 \sum_{m=1}^{\infty}\frac{1}{(2m-1)^{s+1+\gamma}}.
\end{align*}
Since $\sum_{m=1}^\infty (2m-1)^{-p} \le \zeta(p)$ for $p > 1$,
the stated bound follows.
\end{proof}

\subsection{Main Theorem}\label{sec:main_theorem}
 We now state and prove the principal approximation error estimate for the generalized FC-Gram method.
\begin{theorem}
\label{thm:main}
Let $r \in \mathbb{N}\cup \{0\}$ and $f \in C^{r}([0,\,1]) \cap C^{r+1}(0,1)$ with
$f^{(r+1)} \in L^{1}[0,\,1]$.
Let $b \in \mathbb{Q}$, $b > 1$, and suppose there exists $0 \le \beta \le 1$
such that
\begin{equation}
 \label{eq:f-fourier-decay}
 \left|
 \int_{0}^{1} f^{(r+1)}(x)\, e^{-2\pi i\ell x/b}\, \text{d}x
 \right|
 = O\left(|\ell|^{-\beta}\right)
 \quad \text{as } |\ell| \to \infty.
\end{equation}
Let $\mathbb{N}_{b} := \{n \in \mathbb{N} : n \ge d-1, nb \in \mathbb{N},
2 \mid nb\}$.
For $n \in \mathbb{N}_{b}$, let the extension $p = p_{\eta}^{d,n}$ be
constructed as in \eqref{p} with $d \in \mathbb{N}$ Gram
polynomials and a shape-function family $(\eta_{\ell})_{\ell=0}^{d-1}$
satisfying the hypotheses of Section~\ref{sec:GramPolyCont}.
Then there exists a constant $M > 0$, depending on $f$, $r$, $d$, and the
shape-function family but independent of $n$, such that
\begin{equation}
 \label{eq:main-bound}
 \left\|
 \tau_{n}^{b}\left(e[p](f)\right) - f
 \right\|_{\infty,[0,\,1]}
 \le M\, n^{-\min(r+\beta, \, d)}
 \qquad \forall\, n \in \mathbb{N}_{b}.
\end{equation}
\end{theorem}
 
\begin{remark}
\label{rem:beta}
When $\beta = 0$ the Fourier decay condition~\eqref{eq:f-fourier-decay}
is automatically satisfied (it reduces to $f^{(r+1)} \in L^{1}$), and
the bound becomes $O(n^{-\min(r,\, d)})$. If additionally $f^{(r+1)}$ is of bounded variation, then a single
integration by parts shows that its Fourier integral decays as
$O(|\ell|^{-1})$, giving $\beta = 1$ and improving the rate to
$O(n^{-\min(r+1,\, d)})$; see~\cite[Ch.~I]{katznelson2004}.
In either case the rate saturates once $d \geq r + \beta$; increasing
$d$ further does not improve the bound for a given $f$.
\end{remark}
 
\begin{proof}
Set $m := \min(r,\,d-1)$ and $s := \min(r, d)$.
All constants that appear are independent of $\delta$ (equivalently $n$). By the decomposition \eqref{eq:error_decomp}, for every $x \in [0,\,1]$,
\[
 f(x) - \left(\tau_{n}^{b}(e[p]f)\right)(x) = E_{1}(x) + E_{2}(x),
\]
where $E_1$ and $E_2$ are as defined in \eqref{eq:E1E2}.
Each has the form $g(x) - (\tau_{n}^{b}g)(x)$ for a $b$-periodic function
$g$, so it suffices to apply Lemma~\ref{GenFC:lemma_interp_error} to each in turn.

We first verify condition~\eqref{eq:Cond_On_g} for each term.
For $E_1$, take $g_1 = f$ on $[0,\,1]$ and $g_2 = p_{\text{ref}}$ on $(1,b)$.
Condition~\eqref{eq:f-fourier-decay} controls the integral of $g_1^{(s+1)}$
with exponent $\gamma = \beta$ (or $\gamma = 0$ when $d \le r$, in which case
$s = d \le r$ and $f^{(d+1)} \in L^{1}$ suffices).
The first assertion of Lemma~\ref{lemma:sufficientCondEta} gives $O(|\ell|^{-1})$ decay for the integral of $g_2^{(s+1)} = p_{\text{ref}}^{(s+1)}$, and since $\beta \le 1$ the combined bound \eqref{eq:Cond_On_g} holds with $\gamma = \beta$.
For $E_2$, take $g_1 = 0$ and $g_2 = p - p_{\text{ref}}$.
The first integral vanishes, and the second assertion of Lemma~\ref{lemma:sufficientCondEta} provides $O(|\ell|^{-1})$ decay for the integral of $(p - p_{\text{ref}})^{(s+1)}$, so \eqref{eq:Cond_On_g} holds for $E_2$ with $\gamma = 1$.

We now bound $\|E_1\|_{\infty,[0,\,1]}$.
Apply Lemma~\ref{GenFC:lemma_interp_error} to $g = e[p_{\text{ref}}]f$ with
$s = \min(r,d)$ and $\gamma = \beta$.
The jump quantities $J_k$ in \eqref{eq:J_k} involve the one-sided derivatives
of $f$ and $p_{\text{ref}}$ at $x = 0^+/b^-$ and $x = 1^{\pm}$.
By Lemma~\ref{lemma:pref},
$p_{\text{ref}}^{(k)}(1^+) = f^{(k)}(1)$ and
$p_{\text{ref}}^{(k)}(b^-) = f^{(k)}(0)$ for $k = 0, \ldots, m$,
so
\begin{equation}
 \label{eq:E1-Jk-zero}
 J_k = 0 \qquad \text{for } k = 1, \ldots, m.
\end{equation}

\emph{Case~I: $d - 1 < r$.}
Here $s = d$ and $m = d-1$.
Equation~\eqref{eq:E1-Jk-zero} eliminates terms $k = 1, \ldots, d-1$ from
the sum in \eqref{eq:interp_bound}.
For $k = d$ we treat $J_d$ directly. Since $\eta_\ell^R$ and $\eta_\ell^L$
are piecewise $C^{d+2}$, the one-sided limits $p_{\text{ref}}^{(d)}(1^+)$
and $p_{\text{ref}}^{(d)}(b^-)$ exist and are finite. Since $f \in C^r$ with
$r > d-1$, the values $f^{(d)}(1)$ and $f^{(d)}(0)$ are likewise finite.
Hence $J_d$ is finite and the $k = d$ term in \eqref{eq:interp_bound}
contributes $O(n^{-d})$.
Together with the integral term (which is also $O(n^{-d})$ when $\gamma = 0$),
\begin{equation}
 \label{eq:E1-caseI}
 \|E_{1}\|_{\infty,[0,\,1]} \le M_1 n^{-d} = M_1 n^{-\min(r+\beta,\, d)},
\end{equation}
where the equality holds because $d \le r < r + \beta$ in this case.

\emph{Case~II: $d - 1 \ge r$.}
Here $s = r = m$, so \eqref{eq:E1-Jk-zero} eliminates terms $k = 1, \ldots, r$
from the sum, leaving only the integral term in \eqref{eq:interp_bound}.
With $\gamma = \beta$,
\begin{equation}
 \label{eq:E1-caseII}
 \|E_{1}\|_{\infty,[0,\,1]}
 \le M_1 n^{-(r+\beta)} = M_1 n^{-\min(r+\beta,\, d)},
\end{equation}
where the equality holds because $r + \beta \le d$ in this case.

In both cases there exists $M_1 > 0$, independent of $n$, such that
\begin{equation}
 \label{eq:E1-final}
 \|E_{1}\|_{\infty,[0,\,1]} \le M_1 n^{-\min(r+\beta,\, d)}.
\end{equation}

Turning to $E_2$, apply Lemma~\ref{GenFC:lemma_interp_error} to $g_1 \equiv 0$ on
$[0,\,1]$ and $g_2 = p - p_{\text{ref}}$ on $(1,b)$, with $\gamma = 1$.
The jump quantities reduce to
\[
 J_k = \left|(p - p_{\text{ref}})^{(k)}(b^-)\right|
 + \left|(p - p_{\text{ref}})^{(k)}(1^+)\right|.
\]
By Lemmas~\ref{lemma:pref} and~\ref{lemma:p}, for $0 \le k \le m$,
\[
 (p - p_{\text{ref}})^{(k)}(1^+)
 = \left(\frac{2}{\delta}\right)^{k}
 \sum_{\ell=k}^{d-1}
 \langle R_m(f^R), \, p_\ell \rangle\, p_\ell^{(k)}(1),
\]
and analogously at $b^-$.
Applying the remainder bound of Lemma~\ref{lemma:EstimateOnPkRk}
and substituting $\delta = (d-1)/n$,
\begin{equation}\label{eq:E2-Jk-bound}
 J_k \le C_k\,\frac{(d-1)^{m+1-k}}{n^{m+1-k}},
 \qquad k = 0, \dots, m,
\end{equation}
where $C_k > 0$ depends only on $f$, $d$, and the Gram polynomials.

\emph{Case~I: $d-1 < r$.}
Here $s = d$, $m = d-1$, and $\gamma = 1$.
The bound \eqref{eq:E2-Jk-bound} covers $k = 1, \dots, d-1$. For $k = d$, the $d$-th one-sided derivatives of $p_\ell^{L,e}$ and $p_\ell^{R,e}$ at the endpoints exist by the piecewise $C^{d+2}$ assumption
on $\eta$, so $(p - p_{\text{ref}})^{(d)}$ has finite one-sided limits at $1^+$ and $b^-$. Moreover, since these limits are linear combinations of the remainder coefficients $\langle R_m(f^{L,R}), p_\ell \rangle$, each of which satisfies $|\langle R_m(f^L), p_\ell \rangle| \le D_1^{m,\ell}\,\delta^{m+1}
= O(n^{-d})$ by Lemma~\ref{lemma:EstimateOnPkRk} (with $m = d-1$), it follows that $J_d = O(n^{-d})$. The $k=d$ term in \eqref{eq:interp_bound} therefore contributes $O(n^{-2d})$, which is absorbed into $O(n^{-d})$.
Substituting \eqref{eq:E2-Jk-bound} into \eqref{eq:interp_bound},
\begin{align*}
\|E_2\|_{\infty,[0,\,1]} &\le \sum_{k=1}^{d-1}
 \frac{4C_k(d-1)^{d-k}\,\zeta(k+1)}{\pi^{k+1}}\,n^{-d}
 + \frac{4M_2\,\zeta(d+2)}{\pi^{d+2}}\,n^{-(d+1)}\\
 &\le M_2\,n^{-\min(r+\beta,\, d)}, 
\end{align*}
since $d < r \le r + \beta$ in this case.

\emph{Case~II: $d-1 \ge r$.}
Here $s = m = r$ and $\gamma = 1$.
Lemma~\ref{GenFC:lemma_interp_error} involves only $J_k$ for $k = 1, \dots, r$,
all bounded by \eqref{eq:E2-Jk-bound}.
Substituting into \eqref{eq:interp_bound},
\begin{align*}
 \|E_2\|_{\infty,[0,\,1]}
 &\le \sum_{k=1}^{r}
 \frac{4C_k (d-1)^{r+1-k}\,\zeta(k+1)}{\pi^{k+1}}\,n^{-(r+1)}
 + \frac{4M_2\,\zeta(r+2)}{\pi^{r+2}}\,n^{-(r+1)}\\
 &\le M_2\,n^{-(r+\beta)},
\end{align*}
since $n^{-(r+1)} \le n^{-(r+\beta)}$ because $\beta \le 1$.
In both cases there exists $M_2 > 0$, independent of $n$, such that
\begin{equation}\label{eq:E2-final}
 \|E_2\|_{\infty,[0,\,1]} \le M_2\,n^{-\min(r+\beta,\, d)}.
\end{equation}

Setting $M := M_1 + M_2$ and combining \eqref{eq:E1-final} and \eqref{eq:E2-final} gives
\[
 \left\|\tau_n^b(e[p]f) - f\right\|_{\infty,[0,\,1]}
 \le \|E_1\|_{\infty,[0,\,1]} + \|E_2\|_{\infty,[0,\,1]}
 \le M\,n^{-\min(r+\beta,\, d)},
\]
which completes the proof.
\end{proof}

\section{Shape functions}\label{sec:ShapeFunc}
In Section~\ref{sec:GramPolyCont} we have seen that the family of shape
functions $(\eta_\ell^{L}, \eta_\ell^{R})_{\ell=0}^{d-1}$ plays a central
role in constructing the trigonometric interpolant. This section provides
an explicit framework for building such families.

\subsection{Shape functions utilizing two-point Hermite polynomials}
\label{sec:ModFC_shape_function}
A detailed construction of blending-to-zero continuation of Gram 
polynomials using two-point Hermite polynomials is provided in 
\cite{Nainwal_Anand_ModFC_2025}. As the mentioned continuation 
strategy brought several benefits in FC-based approximation strategy, 
we show that the modified FC-Gram (ModFC) turns out to be a particular 
case of GenFC. More precisely, if we consider the following family of 
shape functions:
\begin{align}\label{eta_Hermite}
\eta_\ell^R(x) = \frac{1}{p_{\ell}^R(x)}\sum_{m=0}^{d-1} 
(p_{\ell}^R)^{(m)}(1)p_m^{1, b}(x), \quad
\eta_\ell^L(x) = \frac{1}{p_{\ell}^L(x)}\sum_{m=0}^{d-1} 
(p_{\ell}^L)^{(m)}(b)p_m^{b, 1}(x)
\end{align}
for $0\leq \ell\leq d-1$, where for given positive integer $d$ and 
$x_1 \ne x_2$, we have
\begin{align*}
p_m^{x_1, x_2}(x) &= \frac{1}{m!}(x - x_1)^m 
\left( \frac{x - x_2}{x_1 - x_2} \right)^d 
\sum_{\ell = 0}^{d - (m + 1)} \binom{d + \ell - 1}{d - 1} 
\left( \frac{x - x_1}{x_2 - x_1} \right)^{\ell},
\end{align*}
satisfying
\begin{align*}
\left( p_m^{x_1, x_2} \right)^{(\ell)}(x_1) = \delta_{m\ell}, \quad 
\left( p_m^{x_1, x_2} \right)^{(\ell)}(x_2) = 0, \quad 
0 \leq \ell \leq d - 1.
\end{align*}
The zeros of each Gram polynomial $p_\ell$ lie strictly in $(-1,1)$ 
\cite[Ch.~1]{gautschi2004orthogonal}, and since $\varphi_R$ maps $[1,\,b]$ to values 
outside $[-1,1]$, it follows that $p_\ell^R(x) = p_\ell(\varphi_R(x)) 
\neq 0$ for all $x \in [1,b]$, and analogously $p_\ell^L(x) \neq 0$ 
on $[1,\,b]$. Hence the denominators in \eqref{eta_Hermite} are 
nonzero and the shape functions are well defined. Substituting 
\eqref{eta_Hermite} into \eqref{eq:blendingCont}, the factor 
$p_\ell^R(x)$ cancels and gives
\[
p_\ell^{R,e}(x) = \sum_{m=0}^{d-1}(p_\ell^R)^{(m)}(1)\,p_m^{1,b}(x),
\]
and analogously $p_\ell^{L,e}(x) = 
\sum_{m=0}^{d-1}(p_\ell^L)^{(m)}(b)\,p_m^{b,1}(x)$.
Since $(p_m^{1,b})^{(k)}(1)=\delta_{mk}$ and 
$(p_m^{1,b})^{(k)}(b)=0$ for $0\le k\le d-1$, 
conditions~\eqref{eq:PropertiesOfBlendingCont} are satisfied, and 
the resulting extension~\eqref{p} coincides with the two-point 
Hermite polynomial based blending continuation 
of~\cite{Nainwal_Anand_ModFC_2025}. The shape functions in~\eqref{eta_Hermite} are tied to the Gram
polynomials through the denominators $p_\ell^R$ and $p_\ell^L$, and
depend on $\delta = (d-1)/n$.

We now introduce three families that are constructed independently of
the Gram polynomials and admit $\sigma$ as a free parameter. Given $d$
and $b$, we present the construction of $\eta_\ell^R$ for $0 \leq \ell \leq d-1$; the left counterpart is defined by
\[
 \eta_\ell^L(x) = \eta_\ell^R(b+1-x).
\]
The key idea is to build $\eta_\ell^R$ from a suitably chosen smooth
function $\Phi_\ell$ as
\begin{equation}\label{eq:etaUnified}
 \eta_\ell^R(x;\sigma) =
 \begin{cases}
 1, & x = 1, \\[8pt]
 \Phi_\ell\!\left(\dfrac{x-1}{\sigma-1}\right), & 1 < x \leq \sigma, \\[8pt]
 0, & \sigma < x \leq b,
 \end{cases}
\end{equation}
where $\sigma \in (1,b]$ is a free parameter controlling the width of
the active transition region. We choose $\Phi_\ell$ to be smooth and to
satisfy the flatness conditions
\begin{equation}\label{eq:PhiCond}
 \Phi_\ell(0) = 1, \quad \Phi_\ell(1) = 0, \qquad
 \Phi_\ell^{(k)}(0) = \Phi_\ell^{(k)}(1) = 0,
 \quad k = 1, \ldots, d-1.
\end{equation}
Through the chain rule, these conditions give
$(\eta_\ell^R)^{(k)}(1) = \delta_{k0}$ and
$(\eta_\ell^R)^{(k)}(\sigma^-) = 0$ for $k = 0, \ldots, d-1$, and since
$\eta_\ell^R \equiv 0$ on $(\sigma,b]$, the conditions at $x = b$ hold
trivially. The three families below are each smooth enough that
$\eta_\ell^R$ meets all the requirements given in Section~\ref{sec:GramPolyCont}.

\subsection{Shape functions utilizing the regularized Beta
 function}\label{subsec:ShapeFuncBeta}
Recall the regularized incomplete Beta function
\[
 B_d(s) := \frac{1}{B(d+2,d+2)}\int_0^s t^{d+1}(1-t)^{d+1}\,\mathrm{d}t,
 \qquad s \in [0,1].
\]
Since the integrand is a polynomial of degree $2(d+1)$, $B_d$ is smooth on $[0,1]$ and satisfies $B_d(0) = 0$, $B_d(1) = 1$, and $B_d^{(k)}(0) = B_d^{(k)}(1) = 0$ for $k = 1, \ldots, d+1$. Define
\begin{equation}\label{eq:PhiBeta}
 \Phi^{\mathrm{Beta}}(\xi) := 1 - B_d(\xi).
\end{equation}
Then $\Phi^{\mathrm{Beta}}$ satisfies the flatness
conditions~\eqref{eq:PhiCond}. Since $\Phi^{\mathrm{Beta}}$ is a polynomial, the resulting shape function $\eta_\ell^R$ obtained via~\eqref{eq:etaUnified} is piecewise polynomial and satisfies all the requirements of Section~\ref{sec:GramPolyCont}.

Other smooth families satisfying~\eqref{eq:PhiCond} and the regularity assumptions of Section~\ref{sec:GramPolyCont} can be constructed from standard cutoff functions. Two representative examples are the bump function
\begin{equation}\label{eq:PhiBump}
 \Phi^{\mathrm{Bump}}(\xi)
 := \frac{\varphi(1-\xi)}{\varphi(\xi)+\varphi(1-\xi)},
 \quad
 \varphi(t) = \exp\!\left(-\frac{\log 2}{2t}\right),
 \quad t\in(0,1],
\end{equation}
with $\varphi(0):=0$ and the double-exponential transition
\begin{equation}\label{eq:PhiDExp}
 \Phi^{\mathrm{DExp}}(\xi)
 := \exp\!\left(\frac{2e^{-1/\xi}}{\xi - 1}\right).
\end{equation}
Both belong to $C^\infty([0,\,1])$, satisfy~\eqref{eq:PhiCond} to all orders,
and are independent of $\ell$ and $n$. Blend-to-zero operators on closed intervals, including their construction and connection to smooth transition functions, are studied in~\cite{mendez2026blend}.

\section{Numerical Results} \label{sec:numerics_GenFC}
This section validates the theoretical predictions of
Theorem~\ref{thm:main} through a series of numerical experiments.
According to estimate~\eqref{eq:main-bound}, the convergence rate
$\min(r+\beta,\,d)$ depends on the smoothness $r$ of the target
function, the Fourier-decay exponent $\beta$ defined by
\eqref{eq:f-fourier-decay}, and the number of Gram polynomials $d$. We examine the effect of each parameter individually.

Throughout, the relative approximation error is defined as
\begin{equation}\label{eq:rel-err}
 e_n := \frac{\displaystyle\max_{0\leq j\leq N}
 \left|\tau^b_n f(z_j) - f(z_j)\right|}
 {\displaystyle\max_{0\leq j\leq N}|f(z_j)|},
\end{equation}
where $N = 2^{17}$ and $z_j = j/N$ are evaluation points on a fine
uniform reference grid, and the numerical order of convergence is
\begin{equation}\label{eq:noc}
 \text{noc}_n = \log_2\left(e_{n/2}/e_n\right).
\end{equation}

All experiments are implemented in \textsc{Matlab} and performed on an
Apple M1 laptop. Throughout this section and Section~\ref{sec:GenFC_application}, we use the regularized-Beta shape functions of Section~\ref{subsec:ShapeFuncBeta} as the representative family
for GenFC; the choice is made for uniformity across experiments and any
of the other families in Section~\ref{sec:ShapeFunc} would be equally valid.
We recall that ModFC corresponds to the Hermite shape functions of
Section~\ref{sec:ModFC_shape_function}, which carry no free parameters.
The regularized-Beta family admits a free parameter $\sigma_\ell\in(1,b]$
for each $\ell=0,\ldots,d-1$, written as $\sigma_\ell = 1 +
\widetilde{\sigma}_\ell(b-1)$. Since $\Phi^{\mathrm{Beta}}$ takes
values in $[0,\,1]$, the sup-norm of the blended continuation satisfies
\[
 \|p_\ell^{R,e}\|_{\infty,[1,b]} \leq \|p_\ell^R\|_{\infty,[1,\sigma_\ell]},
\]
and the same bound holds for $p_\ell^{L,e}$. The right-hand side
grows with $\ell$, since $p_\ell^R$ is a polynomial of degree $\ell$
evaluated outside $[-1,1]$, and also grows with $\sigma_\ell$ and $b$.
Reducing $\sigma_\ell$ therefore reduces the sup-norm of the
continuation, with a larger effect for higher $\ell$. In practice,
however, the approximation results are not sensitive to the precise
value of $\sigma_\ell$, and we fix $\widetilde{\sigma}_0 = 1/3$ and
$\widetilde{\sigma}_\ell = 1/10$ for $\ell\geq 1$ throughout all experiments in the paper.

\begin{figure}[htbp]
\centering
\subfloat[$f(x) = \exp\left(\sin(65.5\pi x - 27\pi) - \cos(20.6\pi x)\right)$]%
{\includegraphics[width=0.48\textwidth,
 trim=0cm 0cm 0cm 0cm, clip]
 {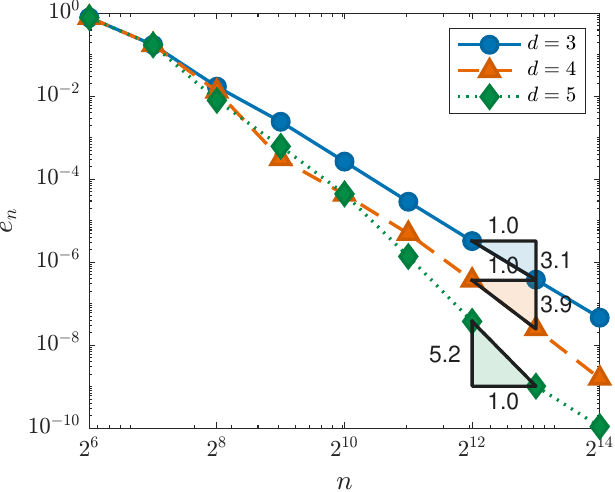}
\label{fig:effectOfd_Case1}}
\hfill
\subfloat[$f(x) = |x - 1/2|^{7/2}$]%
{\includegraphics[width=0.48\textwidth,
 trim=0cm 0cm 0cm 0cm, clip]
{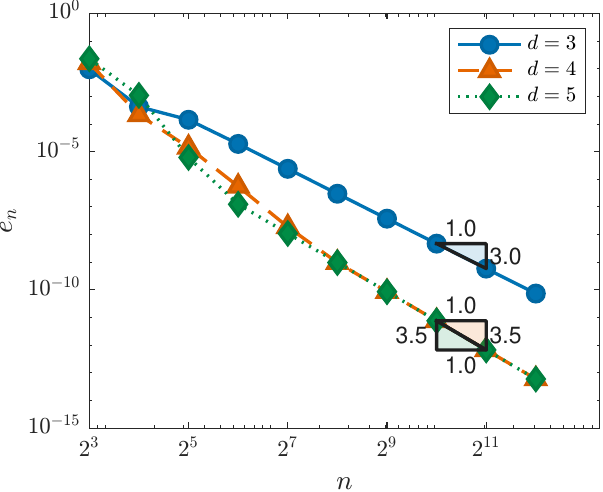}
\label{fig:effectOfd_Case2}}
\caption{log--log convergence plots for the generalized FC-Gram approximation computed with Gram polynomials of degree $d = 3,\, 4,\, 5$ and periodic extension length $b = 2$.}
\label{fig:effectOfd}
\end{figure}

\subsection{Effect of \texorpdfstring{$d$}{d}}
\label{sec:effect-d}
In this section we validate estimate~\eqref{eq:main-bound} by examining
how the number of Gram polynomials $d$ governs the convergence rate.
We present two cases: one where $f$ is infinitely smooth and one where $f$ has limited regularity. For all the experiments in this section the length of periodic continuation is set as $b = 2$.

\begin{figure}[http]
\centering
\subfloat[]%
{\includegraphics[width=0.48\textwidth,
 trim=0cm 0cm 0cm 0cm, clip]
 {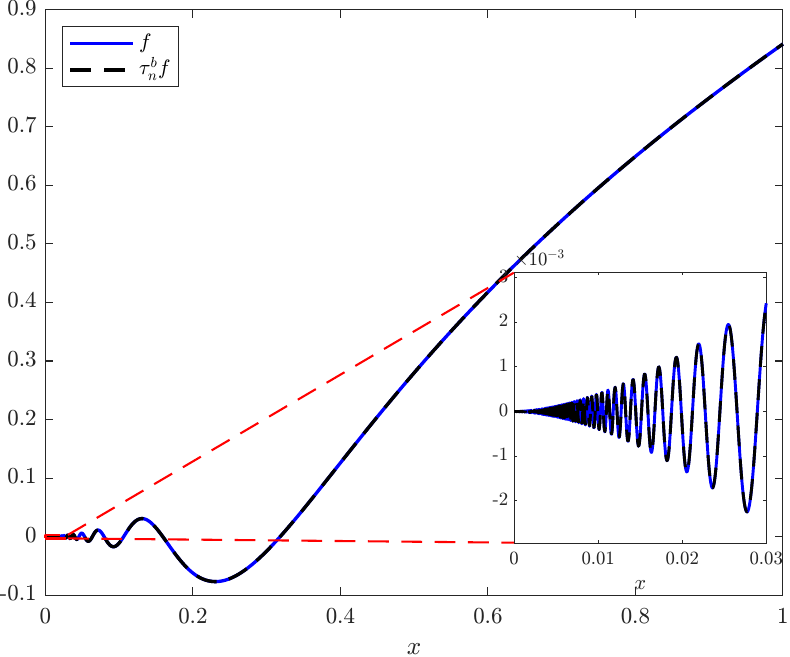}
\label{fig:effectOfd_Case3}}
\hfill
\subfloat[]%
{\includegraphics[width=0.48\textwidth,
 trim=0cm 0cm 0cm 0cm, clip]
{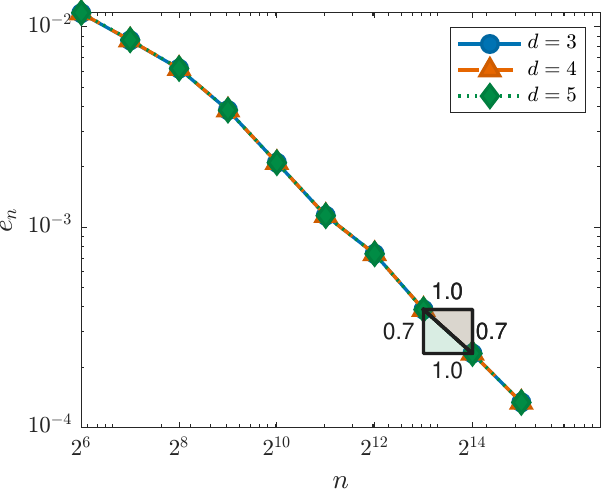}
\label{fig:effectOfd_Case4}}
\caption{Generalized FC-Gram approximation of
$f(x) = x^{1.7}\sin(1/x)$ with $b = 2$.
\textit{Left:} The function $f$ (solid blue) and its approximation
$\tau_n^b f$ (dashed black) on $[0,\,1]$ for $n = 2^{12}$; the inset magnifies
the oscillatory region near $x = 0$.
\textit{Right:} log--log convergence plots for $d = 3,\, 4,\, 5$, asymptotically
attaining the predicted rate $\beta = 0.7$.}
\label{fig:effectOfd_oscillatory}
\end{figure}

We consider $f(x) = \exp\left(\sin(65.5\pi x - 27\pi) - \cos(20.6\pi x)\right)$, an infinitely smooth function. Since $r = \infty$, the expected theoretical rate reduces to $d$, and the accuracy is governed entirely by the number of Gram polynomials used in the boundary projection. Figure~\ref{fig:effectOfd_Case1} clearly illustrates this. The observed $\text{noc}_n$ closely matches $d$ for each of $d = 3$, $4$, and $5$ throughout the computable range. Increasing $d$ consistently and predictably improves the rate, as expected for an infinitely smooth function.

Next, we consider $f(x) = |x - 1/2|^{7/2}$, where $r = 3$ and
$\beta = 1/2$, so the expected rate is $\min(7/2, \,d)$.
Figure~\ref{fig:effectOfd_Case2} shows that the rate increases with $d$
from $d = 3$ to $d = 4$, but for $d = 5$ the asymptotic slope of the
convergence curve matches that of $d = 4$, with both attaining the
predicted rate $\mathrm{noc}_n \approx 7/2$.
This is precisely what Remark~\ref{rem:beta} predicts, as once $d$ exceeds
the regularity threshold $r + \beta = 7/2$, the smoothness of $f$
becomes the binding constraint and increasing $d$ further brings no benefit.

We apply the proposed algorithm to $f(x) = x^{1.7}\sin(1/x)$. This function exhibits increasingly rapid oscillations as $x \to 0^+$, as illustrated in the left panel of Figure~\ref{fig:effectOfd_oscillatory}. It satisfies the decay condition \eqref{eq:f-fourier-decay} with $r = 0$ and $\beta = 0.7$, which yields a predicted theoretical convergence rate of $\min(\beta,\,d) = 0.7$ for all $d \geq 1$. Despite the challenging oscillatory behavior near the origin, the generalized FC-Gram approximation $\tau_n^b f$ with $n = 2^{12}$ and $b = 2$ captures $f$ faithfully on $[0,\,1]$, as seen in the left panel. The right panel of Figure~\ref{fig:effectOfd_oscillatory} displays the log--log convergence curves for $d = 3,\, 4,\, 5$, all of which asymptotically attain a rate close to $0.7$, in close agreement with the theoretical prediction. 

\begin{figure}[http]
\centering
\subfloat[$f(x) = \exp\left(\sin(65.5\pi x - 27\pi) - \cos(20.6\pi x)\right)$]%
{\includegraphics[width=0.48\textwidth,
 trim=0cm 0cm 0cm 0cm, clip]
 {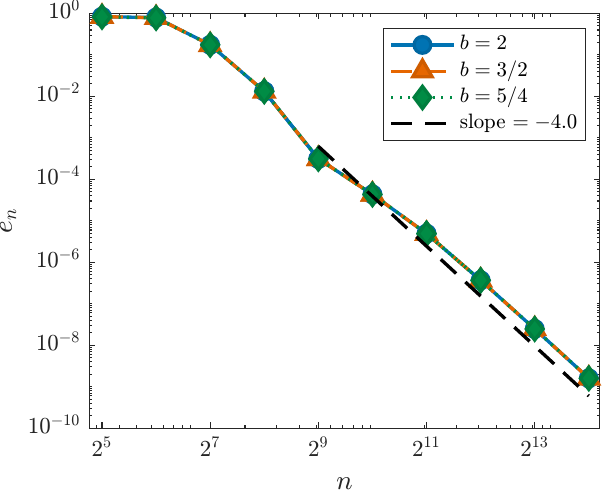}
\label{fig:effectOfb_Case1}}
\hfill
\subfloat[$f(x) = |x - 1/2|^{7/2}$]%
{\includegraphics[width=0.48\textwidth,
 trim=0cm 0cm 0cm 0cm, clip]
 {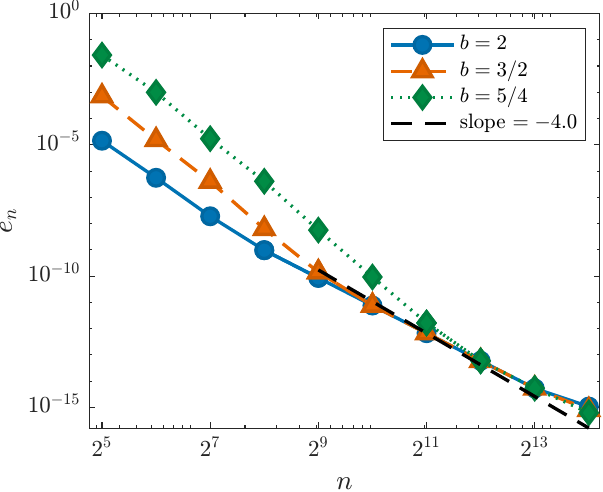}
\label{fig:effectOfb_Case2}}
\caption{log--log convergence plots for the generalized FC-Gram approximation
computed with Gram polynomials of degree $d = 4$ and periodic extension length
$b = 2,\, 1.5,\, 1.25$.}
\label{fig:effectOfb}
\end{figure}

We also verify that the asymptotic rate is robust to the choice of $b$.
Fixing $d = 4$ and varying $b \in \{2,\, 3/2,\, 5/4\}$, the predicted 
rates are confirmed for the smooth function and $f(x) = |x-1/2|^{7/2}$; 
see Figures~\ref{fig:effectOfb_Case1} and~\ref{fig:effectOfb_Case2}.
For $f(x) = x^{1.7}\sin(1/x)$, the convergence rate is already 
confirmed in Figure~\ref{fig:effectOfd_oscillatory} and the 
independence from $b$ follows from the same argument. This confirms that the asymptotic rate $\min(r+\beta,\,d)$ established in Theorem~\ref{thm:main} is independent of the extension length $b$.

\subsection{Effect of $r+\beta$}
We examine how the regularity of $f$ and the Fourier-decay exponent 
$\beta$ each contribute to the convergence rate. Throughout this 
section we use $f(x) = (1-x)^{r+\beta}$, for which both quantities 
are determined exactly by the choice of exponent, and fix $d = 5$ and 
$b = 2$ so that $\min(r+\beta,\, d) = r+\beta$.

First, we fix $r = 3$ and vary $\beta = 0.2,\, 0.4,\, 0.8$, for which 
the predicted rates are $3.2,\, 3.4,\, 3.8$; the observed rates in 
Figure~\ref{fig:effectOfbeta} are in close agreement. Next, we fix 
$\beta = 0.5$ and take $r = 2,\, 3,\, 4$, giving predicted rates 
$2.5,\, 3.5,\, 4.5$, which are confirmed in Figure~\ref{fig:effectOfr}. 
In both cases the rate is governed by the sum $r+\beta$, in agreement 
with Theorem~\ref{thm:main}.

\begin{figure}[htbp]
\centering
\subfloat[]%
{\includegraphics[width=0.48\textwidth,
 trim=0cm 0cm 0cm 0cm, clip]
 {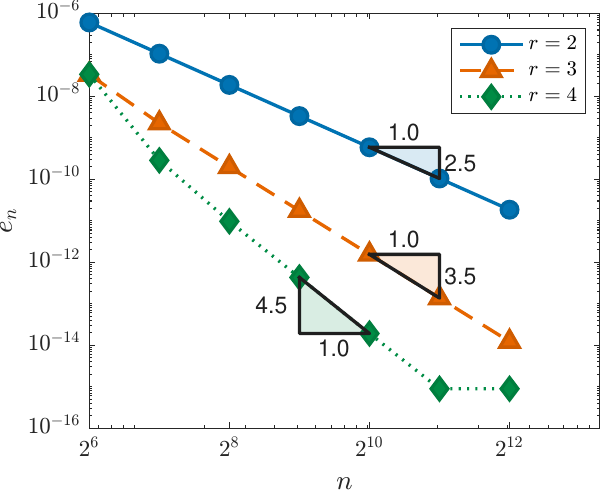}
\label{fig:effectOfr}}
\hfill
\subfloat[]%
{\includegraphics[width=0.48\textwidth,
 trim=0cm 0cm 0cm 0cm, clip]
 {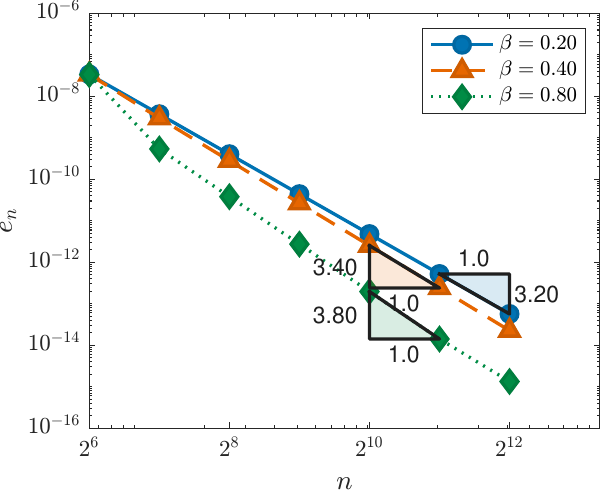}
\label{fig:effectOfbeta}}
\caption{log--log convergence plots for the generalized FC-Gram approximation of $f(x) = (1-x)^{3+\beta}$ using Gram polynomials of degree $d = 5$, periodic extension length $b = 2$, and varying $\beta$.}
\label{fig:effectOfrPlusbeta}
\end{figure}

Finally, we present an example illustrating that although ModFC and
GenFC achieve the same asymptotic rate, GenFC yields improved
approximation accuracy. Consider $f(x) = \exp(-\cos(kx))$ for
$k = 100,\, 200,\, 300$. This function is infinitely smooth, so the
asymptotic convergence rate is $d$ for both algorithms. As $k$
increases, the function oscillates more rapidly near the boundary,
causing the Gram projection coefficients to grow in magnitude and
thereby increasing the sup-norm of the blended continuation on $[1,\,b]$.
The free parameter $\sigma_\ell$ in GenFC allows one to reduce this
sup-norm, and the effect on the approximation error is visible in
Figure~\ref{fig:ModFC_Vs_GenFC_exp_minus_cos_paramx}: the gain in
accuracy widens with $k$, while the asymptotic rate $d$ remains the
same for both methods. This advantage carries over directly to
differential equation solvers, as demonstrated in Section~\ref{sec:GenFC_application}.
\begin{figure}[htbp]
\centering
\subfloat[$k = 100$]%
{\includegraphics[width=0.325\textwidth,
 trim={0pt 0pt 0pt 0pt}, clip]
{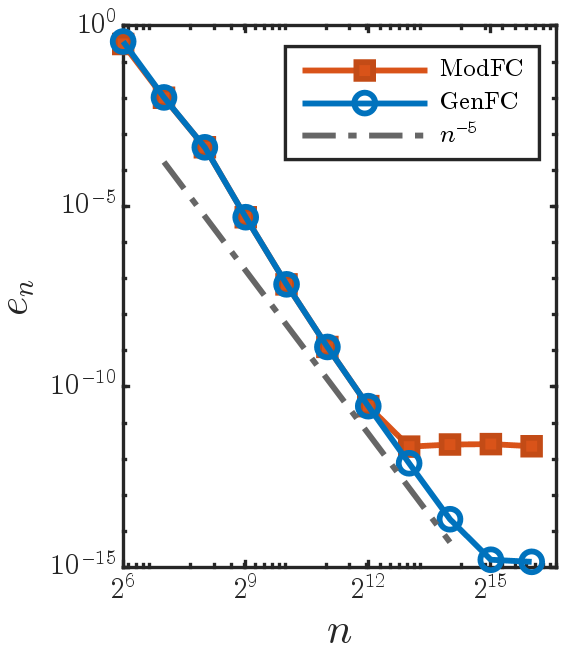}}
\hfill
\subfloat[$k = 200$]%
{\includegraphics[width=0.325\textwidth,
 trim={0pt 0pt 0pt 0pt}, clip]
{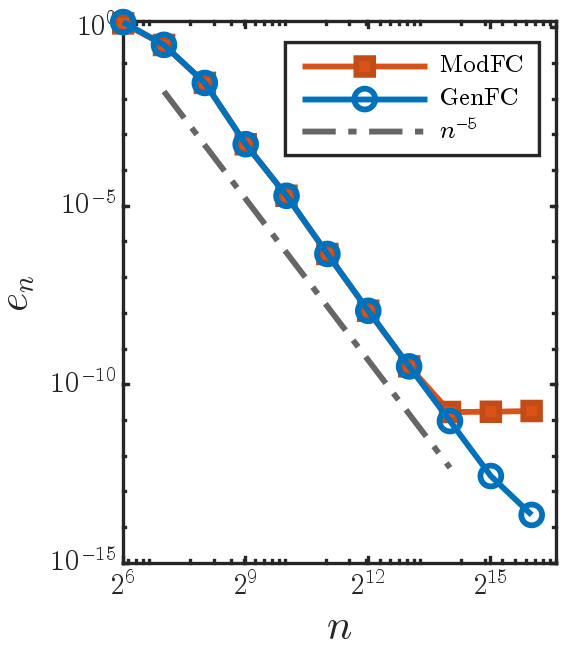}}
\hfill
\subfloat[$k = 300$]%
{\includegraphics[width=0.325\textwidth,
 trim={0pt 0pt 0pt 0pt}, clip]
{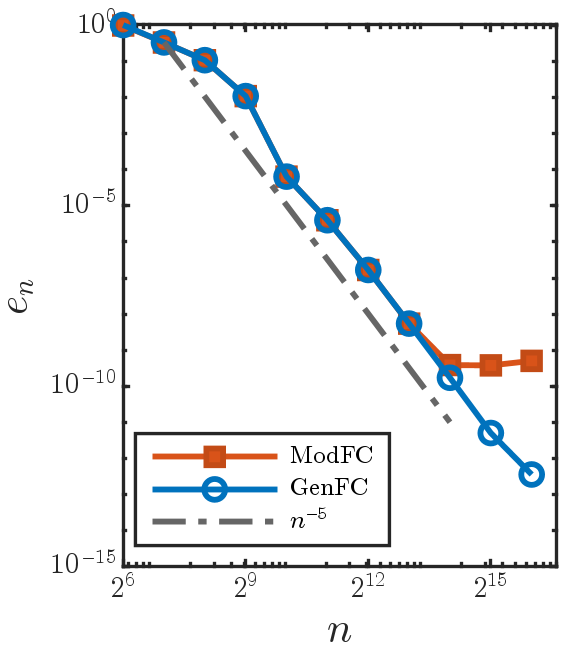}}
\caption{Interpolation errors for ModFC and GenFC applied to  $f(x)=\exp(-\cos(kx))$, with $d=5$ and $b=2$, for $k=100,200,300$.}
\label{fig:ModFC_Vs_GenFC_exp_minus_cos_paramx}
\end{figure}

\section{Application to differential equations}
\label{sec:GenFC_application}
We demonstrate that the approximation advantages of GenFC carry 
over to differential equation solvers. We begin with a spectral 
solver for two-point boundary value problems in Section~\ref{subsec:2pointBVP}, which is subsequently employed 
at each time step of the BDF based parabolic solver developed 
in Section~\ref{subsec:pde}.

\subsection{Two-Point Boundary Value Problems}
\label{subsec:2pointBVP}
Throughout this section we take $b = 2$ and $d = 5$, and write $\tau_n(f) := \tau_n^2(f)$ for brevity. We consider the two-point boundary value problem
\begin{subequations}\label{eq:bvp}
\begin{align}
 u''(x) + P(x)u'(x) + Q(x)u(x) + R(x) &= 0,
 \quad x\in(0,1), \label{eq:bvp-ode}\\
 a_0 u(0) - b_0 u'(0) = c_0,\qquad
 a_1 u(1) + b_1 u'(1) &= c_1, \label{eq:bvp-bc}
\end{align}
\end{subequations}
and seek an approximate solution of the form
\begin{equation}\label{eq:vn}
 v_n(x) = \sum_{\ell=-n}^{n-1} v_{n,\ell}\, e^{\pi i \ell x}.
\end{equation}
The coefficients $v_{n,\ell}$ in~\eqref{eq:vn} are 
determined by first constructing periodic extensions $P^e$, 
$Q^e$, $R^e$ of $P$, $Q$, $R$ respectively via the GenFC 
algorithm, and then requiring $v_n$ to satisfy
\begin{equation}\label{eq:ode-periodic}
 v_n''(x) + \tau_n(P^e)(x)\,v_n'(x)
 + \tau_n(Q^e)(x)\,v_n(x)
 + \tau_n(R^e)(x) = 0, \quad x\in\mathbb{R}.
\end{equation}
Substituting~\eqref{eq:vn} into~\eqref{eq:ode-periodic} and collecting Fourier modes leads to the rectangular linear system
\begin{align*}
 \sum_{\ell=-n}^{n-1}
 \left[(\pi i\ell)\,\tilde{c}_{k-\ell}(P^e)
 + \tilde{c}_{k-\ell}(Q^e)\right]v_{n,\ell}
 &= 0,
 \quad -2n \le k \le -n-1, \nonumber\\
 -(\pi k)^2 v_{n,k}
 +\sum_{\ell=-n}^{n-1}
 \left[(\pi i\ell)\,\tilde{c}_{k-\ell}(P^e)
 + \tilde{c}_{k-\ell}(Q^e)\right]v_{n,\ell}
 &= -\tilde{c}_k(R^e),
 \quad -n \le k \le n-1, \\
 \sum_{\ell=-n}^{n-1}
 \left[(\pi i\ell)\,\tilde{c}_{k-\ell}(P^e)
 + \tilde{c}_{k-\ell}(Q^e)\right]v_{n,\ell}
 &= 0,
 \quad n \le k \le 2n-1, \nonumber
\end{align*}
where $\tilde{c}_m(P^e),\, \tilde{c}_m(Q^e)$ and $\tilde{c}_m(R^e)$ are zero-padded for $|m|\ge n$, and the system is solved in the least-squares sense. Moreover, with reference to \eqref{eq:c_ell_g}, we have
\begin{align*}
 \tilde{c}_\ell(g) = \frac{1}{nb}\sum_{j=0}^{nb-1} g(x_j)e^{-2\pi i \ell x_j/b}.
\end{align*}
Since $v_n$ is periodic, it generally cannot satisfy the boundary conditions~\eqref{eq:bvp-bc} on its own. The final approximate solution is therefore constructed as
\begin{equation}\label{eq:un}
 u_n(x) = v_n(x) + \xi_1 h_1(x) + \xi_2 h_2(x),
\end{equation}
where $h_1$ and $h_2$ are two linearly independent solutions of the homogeneous equation~\eqref{eq:bvp-ode} with $R\equiv 0$, and the scalars $\xi_1$, $\xi_2$ are chosen so that $u_n$ satisfies~\eqref{eq:bvp-bc} exactly.
The role of $h_1$ and $h_2$ is to absorb the boundary discrepancy left by the periodic part $v_n$, while the approximation accuracy is determined entirely by the quality of the FC continuations $P^e$, $Q^e$, $R^e$ entering~\eqref{eq:ode-periodic}. A similar boundary-correction strategy is used in~\cite{BrunoPrieto2014}. The approximation error in this section is given by
\begin{equation*}
 e_n := \frac{\displaystyle\max_{0\leq j\leq n}
 \left|u_n(j/n) - u(j/n)\right|}
 {\displaystyle\max_{0\leq j\leq n}|u(j/n)|},
\end{equation*}
and $\mathrm{noc}_n$ is as defined in~\eqref{eq:noc}.

\begin{table}[htbp]
\centering
\begin{tabular}{c|cc|cc|cc}
\hline
\multirow{2}{*}{$n$}
 & \multicolumn{2}{c|}{$k=100$}
 & \multicolumn{2}{c|}{$k=200$}
 & \multicolumn{2}{c}{$k=300$} \\
\cline{2-7}
 & $e_n$ & $\mathrm{noc}_n$
 & $e_n$ & $\mathrm{noc}_n$
 & $e_n$ & $\mathrm{noc}_n$ \\
\hline\hline
$2^{6}$ & $1.93\times10^{-2}$ & ---
 & $1.05\times10^{0}$ & ---
 & $7.50\times10^{0}$ & --- \\
$2^{7}$ & $3.05\times10^{-4}$ & $5.99$
 & $1.63\times10^{-2}$ & $6.01$
 & $2.53\times10^{-1}$ & $4.89$ \\
$2^{8}$ & $2.45\times10^{-6}$ & $6.96$
 & $2.97\times10^{-4}$ & $5.78$
 & $4.51\times10^{-3}$ & $5.81$ \\
$2^{9}$ & $1.60\times10^{-8}$ & $7.26$
 & $2.58\times10^{-6}$ & $6.85$
 & $3.58\times10^{-5}$ & $6.98$ \\
$2^{10}$ & $1.06\times10^{-10}$ & $7.23$
 & $2.04\times10^{-8}$ & $6.98$
 & $3.06\times10^{-7}$ & $6.87$ \\
$2^{11}$ & $7.83\times10^{-13}$ & $7.09$
 & $1.53\times10^{-10}$ & $7.06$
 & $2.71\times10^{-9}$ & $6.82$ \\
$2^{12}$ & $1.47\times10^{-13}$ & $2.42$
 & $7.53\times10^{-12}$ & $4.35$
 & $1.11\times10^{-10}$ & $4.61$ \\
\hline
\end{tabular}
\caption{Convergence study of the generalized FC-Gram approximation applied to the BVP~\eqref{eq:bvp-coskx} with $\lambda = 0.1$ and $k = 100$, $200$, $300$.}
\label{tab:ode_genfc_coskx}
\end{table}

\begin{example}
Consider the singularly perturbed problem
\begin{equation}\label{eq:bvp-coskx}
 -\lambda\, u''(x) + u(x) = \cos(kx),
 \quad x\in(0,1),\quad u(0)=u(1)=0,
\end{equation}
for $\lambda = 0.1$ and $k = 100,\, 200,\, 300$. 
Table~\ref{tab:ode_genfc_coskx} reports the errors and convergence rates. For all three values of $k$ the method attains rates close to $d + 2 = 7$, with the errors decreasing consistently even for $k = 300$, where the source term oscillates rapidly near the boundary.

\end{example}

\begin{example}
 We consider a variable-coefficient problem with a near-singularity at the left boundary:
\begin{equation}\label{eq:bvp-xpluseps}
 (x+\varepsilon)^2 u''(x) + 2(x+\varepsilon)u'(x) - 2u(x)
 = \sin(\log(x+\varepsilon)),
 \quad u(0)=1,\ u(1)=2,
\end{equation}
for $\varepsilon = 1/5$, $1/10$, $1/20$. As $\varepsilon$
decreases, the coefficient functions become increasingly steep near $x = 0$, making this a challenging test for the FC continuation. Table~\ref{tab:ode_genfc_xpluseps} shows that GenFC attains rates close to $d+2$ and reaches near machine precision in all three cases.
\begin{table}[htbp]
\centering
\begin{tabular}{c|cc|cc|cc}
\hline
\multirow{2}{*}{$n$}
 & \multicolumn{2}{c|}{$\varepsilon=1/5$}
 & \multicolumn{2}{c|}{$\varepsilon=1/10$}
 & \multicolumn{2}{c}{$\varepsilon=1/20$} \\
\cline{2-7}
 & $e_n$ & $\mathrm{noc}_n$
 & $e_n$ & $\mathrm{noc}_n$
 & $e_n$ & $\mathrm{noc}_n$ \\
\hline\hline
$2^{6}$ & $8.45\times10^{-9}$ & --- & $1.59\times10^{-6}$ & --- & $5.12\times10^{-5}$ & --- \\
$2^{7}$ & $4.82\times10^{-11}$ & $7.45$ & $1.19\times10^{-8}$ & $7.05$ & $7.45\times10^{-7}$ & $6.10$ \\
$2^{8}$ & $7.03\times10^{-13}$ & $6.10$ & $1.49\times10^{-10}$ & $6.32$ & $1.29\times10^{-8}$ & $5.85$ \\
$2^{9}$ & $6.44\times10^{-15}$ & $6.77$ & $1.52\times10^{-12}$ & $6.62$ & $1.62\times10^{-10}$ & $6.32$ \\
$2^{10}$ & $3.33\times10^{-16}$ & $4.27$ & $1.35\times10^{-14}$ & $6.81$ & $1.62\times10^{-12}$ & $6.64$ \\
$2^{11}$ & $3.33\times10^{-16}$ & $0.00$ & $4.44\times10^{-16}$ & $4.93$ & $1.45\times10^{-14}$ & $6.81$ \\
\hline
\end{tabular}
\caption{Convergence study of the generalized FC-Gram approximation applied to
the BVP~\eqref{eq:bvp-xpluseps} for $\varepsilon = 1/5$, $1/10$, $1/20$.}
\label{tab:ode_genfc_xpluseps}
\end{table} 
\end{example}

\begin{figure}[htbp]
\centering
\subfloat[]%
{\includegraphics[width=0.48\textwidth,
 trim=0cm 0cm 0cm 0cm, clip]
 {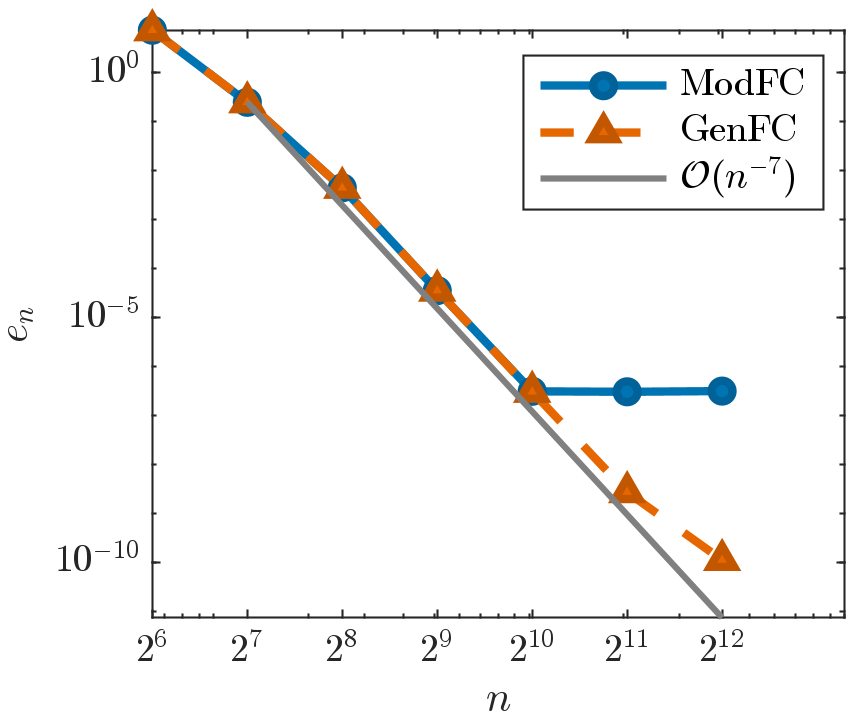}
\label{fig:ModFcVsGenFCbvpEx1}}
\hfill
\subfloat[]%
{\includegraphics[width=0.48\textwidth,
 trim=0cm 0cm 0cm 0cm, clip]
 {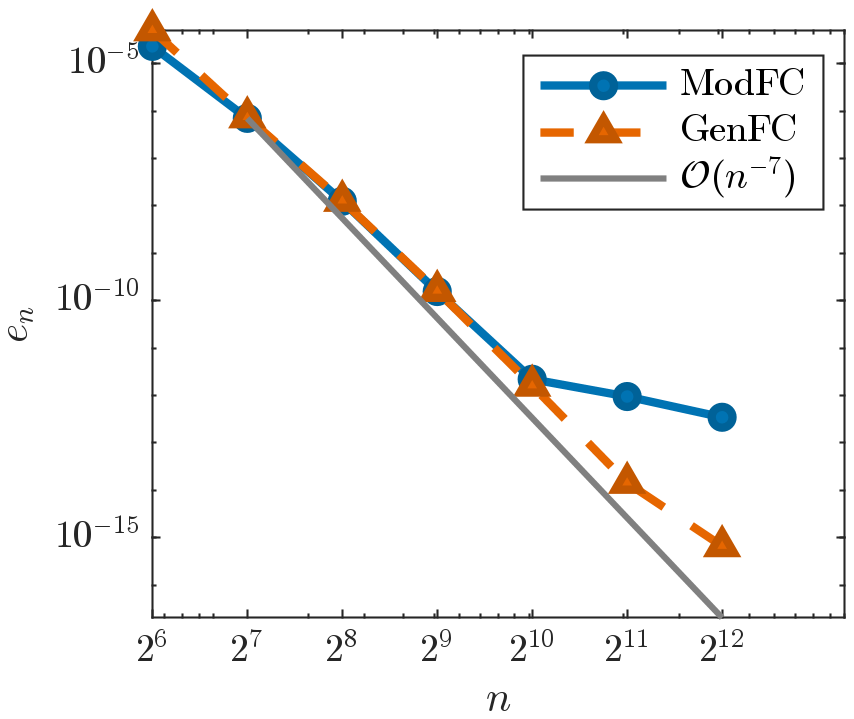}
\label{fig:ModFcVsGenFCbvpEx2}}
\caption{Convergence study of the BVP solver with modified and generalized FC-Gram with $d = 5$ and $b = 2$. Left: BVP \eqref{eq:bvp-coskx} with $\lambda = 0.1$ and $k = 300$. Right: BVP \eqref{eq:bvp-xpluseps} with $\varepsilon = 1/20$.}
\label{fig:ModFcVsGenFCbvp}
\end{figure}

Figure~\ref{fig:ModFcVsGenFCbvp} compares ModFC and GenFC on the two BVP examples. Figure~\ref{fig:ModFcVsGenFCbvpEx1} corresponds to~\eqref{eq:bvp-coskx} with $k = 300$ and Figure~\ref{fig:ModFcVsGenFCbvpEx2} corresponds to~\eqref{eq:bvp-xpluseps} with $\varepsilon = 1/20$. Although both methods achieve the same convergence rate, GenFC gains several additional digits of accuracy over ModFC, demonstrating the practical benefit of the reduced sup-norm of the blended continuation.
\begin{table}[htbp]
\centering
\begin{tabular}{lrrrrr}
\hline
 & $n$ & $n_{T}$ & $\Delta t$ &
$e_n$ & $\mathrm{noc}_n$ \\
\hline
BDF-2 & 8 & 32 & 1.56e$-$02 & 7.67e$-$03 & --- \\
 & 16 & 128 & 3.91e$-$03 & 5.15e$-$04 & 3.90 \\
 & 32 & 512 & 9.77e$-$04 & 3.75e$-$05 & 3.78 \\
 & 64 & 2048 & 2.44e$-$04 & 2.45e$-$06 & 3.93 \\
 & 128 & 8192 & 6.10e$-$05 & 1.55e$-$07 & 3.98 \\
\hline
BDF-3 & 8 & 32 & 1.56e$-$02 & 8.84e$-$03 & --- \\
 & 16 & 128 & 3.91e$-$03 & 7.78e$-$05 & 6.83 \\
 & 32 & 512 & 9.77e$-$04 & 1.80e$-$06 & 5.44 \\
 & 64 & 2048 & 2.44e$-$04 & 2.37e$-$08 & 6.24 \\
 & 128 & 8192 & 6.10e$-$05 & 5.85e$-$10 & 5.34 \\
\hline
\end{tabular}
\caption{Convergence study for the heat equation~\eqref{eq:adv-diff} ($a = 0$, $\nu = 1$) under the parabolic CFL $\Delta t = h^2$, using the BVP solver based on GenFC with $d = 5$ and $b = 2$ at each time step together with BDF-$k$ time integration.}
\label{tab:heat-parabolic}
\end{table}

\begin{table}[H]
\centering
\begin{tabular}{lrrrrr}
\hline
 & $n$ & $n_{T}$ & $\Delta t$ &
$e_n$ & $\mathrm{noc}_n$ \\
\hline
BDF-3 & 8 & 32 & 1.56e$-$02 & 8.84e$-$03 & --- \\
($\Delta t=0.5h^{5/3}$) & 16 & 102 & 4.90e$-$03 & 1.24e$-$04 & 6.16 \\
 & 32 & 323 & 1.55e$-$03 & 4.14e$-$06 & 4.90 \\
 & 64 & 1024 & 4.88e$-$04 & 1.44e$-$07 & 4.84 \\
 & 128 & 3251 & 1.54e$-$04 & 4.68e$-$09 & 4.94 \\
\hline
BDF-4 & 8 & 13 & 3.85e$-$02 & 2.39e$-$02 & --- \\
($\Delta t=0.5h^{5/4}$) & 16 & 32 & 1.56e$-$02 & 2.24e$-$04 & 6.74 \\
 & 32 & 76 & 6.58e$-$03 & 6.20e$-$06 & 5.17 \\
 & 64 & 181 & 2.76e$-$03 & 2.78e$-$07 & 4.48 \\
 & 128 & 431 & 1.16e$-$03 & 9.74e$-$09 & 4.83 \\
\hline
BDF-5 & 8 & 8 & 6.25e$-$02 & 7.79e$-$02 & --- \\
($\Delta t=0.5h$, & 16 & 16 & 3.12e$-$02 & 4.77e$-$03 & 4.03 \\
hyperbolic) & 32 & 32 & 1.56e$-$02 & 1.88e$-$04 & 4.66 \\
 & 64 & 64 & 7.81e$-$03 & 5.76e$-$06 & 5.03 \\
 & 128 & 128 & 3.91e$-$03 & 1.64e$-$07 & 5.14 \\
\hline
\end{tabular}
\caption{Convergence study for the heat equation~\eqref{eq:adv-diff} 
($a = 0$, $\nu = 1$) under the balanced CFL $\Delta t = 0.5\,h^{d/k}$, using the BVP solver based on GenFC with $d = 5$ and $b = 2$ at each time step together with BDF-$k$ time integration.}
\label{tab:heat-balanced}
\end{table}

\subsection{Application to parabolic problems}
\label{subsec:pde}
High-order spatial discretizations are most effective when paired with time-stepping schemes of matching order, since a low-order time integrator will limit the overall accuracy regardless of the spatial resolution. The BVP solver based on GenFC developed in Section~\ref{subsec:2pointBVP} provides high-order spatial accuracy at each time level, and couples naturally with BDF schemes since the implicit time discretization reduces the parabolic problem at each step to a BVP of the form~\eqref{eq:bvp-ode}, which the BVP solver based on GenFC handles directly. We demonstrate this on the advection-diffusion equation
\begin{equation}\label{eq:adv-diff}
 u_t + a\,u_x = \nu\,u_{xx} + f(x,t),
 \quad (x,t)\in(0,1)\times(0,T],
\end{equation}
with Dirichlet data $u(0,t)=g_L(t)$, $u(1,t)=g_R(t)$, and initial condition $u(x,0)=u_0(x)$. The BDF-$k$ formula with step size $\Delta t$ and coefficients $\{\alpha_j\}_{j=0}^k$ reduces~\eqref{eq:adv-diff} at each time level to the BVP
\begin{equation}\label{eq:bdf-bvp}
 -\nu\,u^{n+1}_{xx} + a\,u^{n+1}_x
 + \frac{\alpha_0}{\Delta t}\,u^{n+1}
 = -\sum_{j=1}^{k}\frac{\alpha_j}{\Delta t}\,u^{n+1-j}
 + f(\cdot,\,t_{n+1}),
\end{equation}
with boundary conditions $u^{n+1}(0)=g_L(t_{n+1})$ and $u^{n+1}(1)=g_R(t_{n+1})$. This is of the form~\eqref{eq:bvp-ode} with constant coefficients $P = -a/\nu$ and $Q = -\alpha_0/(\nu\Delta t)$, so the null-space basis is explicit:
\begin{equation}\label{eq:null-adv}
 h_1(x) = e^{\rho_1 (x-1)},\quad h_2(x) = e^{\rho_2 x},
 \qquad
 \rho_{1,2} = \tfrac{1}{2}\!\left(\tfrac{a}{\nu}
 \pm\sqrt{\bigl(\tfrac{a}{\nu}\bigr)^2
 +\tfrac{4\alpha_0}{\nu\Delta t}}\right).
\end{equation}
We assume $a \ge 0$ without loss of generality; the case $a < 0$ follows by the reflection $x \mapsto 1-x$.

Each time step requires a single solve at cost $O(n\log n)$, and the startup values $u^1, \ldots, u^{k-1}$ are computed by a DIRK scheme of matching order, with the BVP solver 
based on GenFC used at each stage. The error is measured by
\begin{equation*}
 e_n := \frac{\displaystyle\max_{0\leq j\leq n}
 \left|u_n(j/n,T) - u(j/n,T)\right|}
 {\displaystyle\max_{0\leq j\leq n}|u(j/n,T)|},
\end{equation*}
and $\mathrm{noc}_n$ is as defined in~\eqref{eq:noc}. The periodic extension length $b = 2$ is used throughout in the GenFC based PDE solver.

\begin{example}\label{Problem:heat}
We take $a = 0$, $\nu = 1$, and $T = 0.5$, with the manufactured solution
\begin{equation*}
 u(x,t) = \cos(15t)\sin(5x+5),
\end{equation*}
and Dirichlet data $g_L(t) = \cos(15t)\sin 5$, $g_R(t) = \cos(15t)\sin 10$.

Tables~\ref{tab:heat-parabolic} and~\ref{tab:heat-balanced} 
report the errors in the supremum norm under two CFL scalings. We invoke the BVP solver with $d = 5$. Under the parabolic CFL $\Delta t = \,h^2$, BDF-2 achieves rate $4$ and BDF-3 achieves rate $5$, with higher BDF orders giving the same rate as BDF-3. Under the balanced
CFL $\Delta t = 0.5\,h^{d/k}$, rate $5$ is observed for BDF-3, BDF-4, and BDF-5. The step count scales as $n^{d/k}$, decreasing as $k$ increases. At BDF-5 the CFL exponent gives $n_{\text{steps}} = n$, with error $1.64 \times 10^{-7}$ at $n = 128$ requiring exactly $n$ time steps.
\end{example}

\begin{example}\label{Problem:adv-diff}
We take $a = 1$, $\nu = 10^{-3}$, and $T = 1$, giving a convection-dominated problem with P\'eclet number $a/\nu = 10^3$. The manufactured solution is
\begin{equation*}
 u(x,t) = \cos(5t)\sin(10x+10),
\end{equation*}
with Dirichlet data $g_L(t) = \cos(5t)\sin 10$ and $g_R(t) = \cos(5t)\sin 20$.

\begin{figure}[H]
\centering
\subfloat[]%
{\includegraphics[width=0.48\textwidth,
 trim=0cm 0cm 0cm 0cm, clip]
 {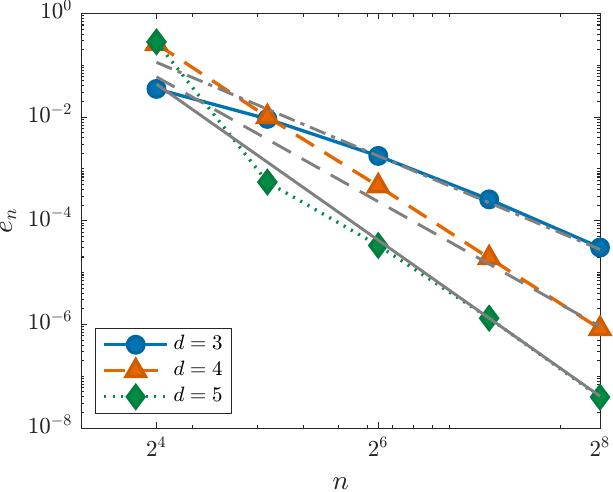}
\label{fig:adv_diff_convergence_parabolic}}
\hfill
\subfloat[]%
{\includegraphics[width=0.48\textwidth,
 trim=0cm 0cm 0cm 0cm, clip]
 {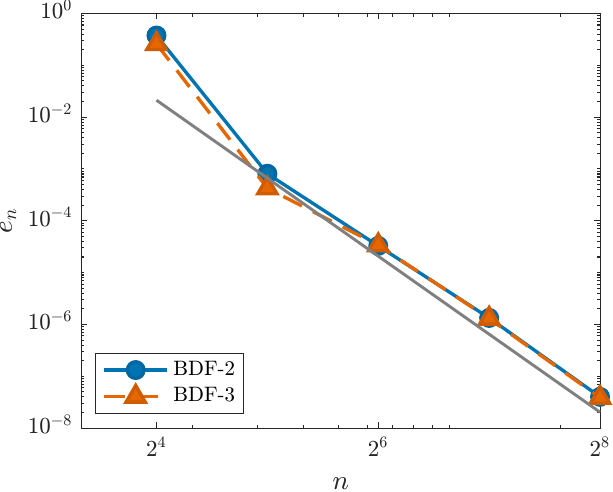}
\label{fig:adv_diff_convergence_balanced}}
\caption{Convergence study for the advection-diffusion equation~\eqref{eq:adv-diff} with $a = 1$ and $\nu = 10^{-3}$, using the BVP solver based on GenFC at each time step. Left: parabolic CFL $\Delta t = \,h^2$ with BDF-$3$ and varying $d = 3,\, 4,\, 5$. Right: balanced CFL $\Delta t = 0.5\,h^{d/k}$ with $d = 5$ and BDF-$k$ for $k = 2, \, 3$.}
\label{fig:adv_diff_convergence}
\end{figure}
Figure~\ref{fig:adv_diff_convergence_parabolic} shows the errors in the supremum norm under the parabolic CFL $\Delta t = h^2$ with BDF-$3$ time integration and $d = 3,\, 4,\, 5$. The observed rates are $3$, $4$, and $5$ respectively, that is, a spatial rate of $d$, the number of Gram polynomials used by the GenFC based BVP solver. 
Figure~\ref{fig:adv_diff_convergence_balanced} shows the corresponding results under the balanced CFL $\Delta t = 0.5\,h^{d/k}$ with $d = 5$ and $k = 2,\,3$, where the observed rate is $5$ in both cases. The proposed PDE solver thus attains high-order accuracy throughout, without any filtering.

\end{example}

\section{Conclusion}
\label{sec:conclusion}
In this paper we have studied GenFC, a generalized FC-Gram approximation framework. A key ingredient of the FC-Gram strategy is the construction of blending continuations of Gram polynomials over the extension interval $[1,b]$, and the quality of these continuations directly governs the approximation accuracy. GenFC provides explicit control over this construction through a broad family of shape functions.

In Theorem~\ref{thm:main} we prove convergence at the rate $O(n^{-\min(r+\beta,\,d)})$ in the supremum norm on $[0,\,1]$. The estimate applies to any admissible shape-function family and any $f$ of the required regularity. The modified FC-Gram method of~\cite{Nainwal_Anand_ModFC_2025} is recovered as a special case. Numerical experiments in a variety of settings confirm the theoretical predictions.

The flexibility in the choice of shape function leads to improved accuracy over ModFC, particularly when the boundary data are large, and this advantage carries over to high-order solvers for linear ODEs in one dimension. This BVP solver is subsequently used at each time step of a BDF scheme for parabolic PDEs, and high-order accuracy is demonstrated for the heat and advection-diffusion equations. 
\begin{figure}[H]
\centering
\subfloat{%
 \includegraphics[width=3in]{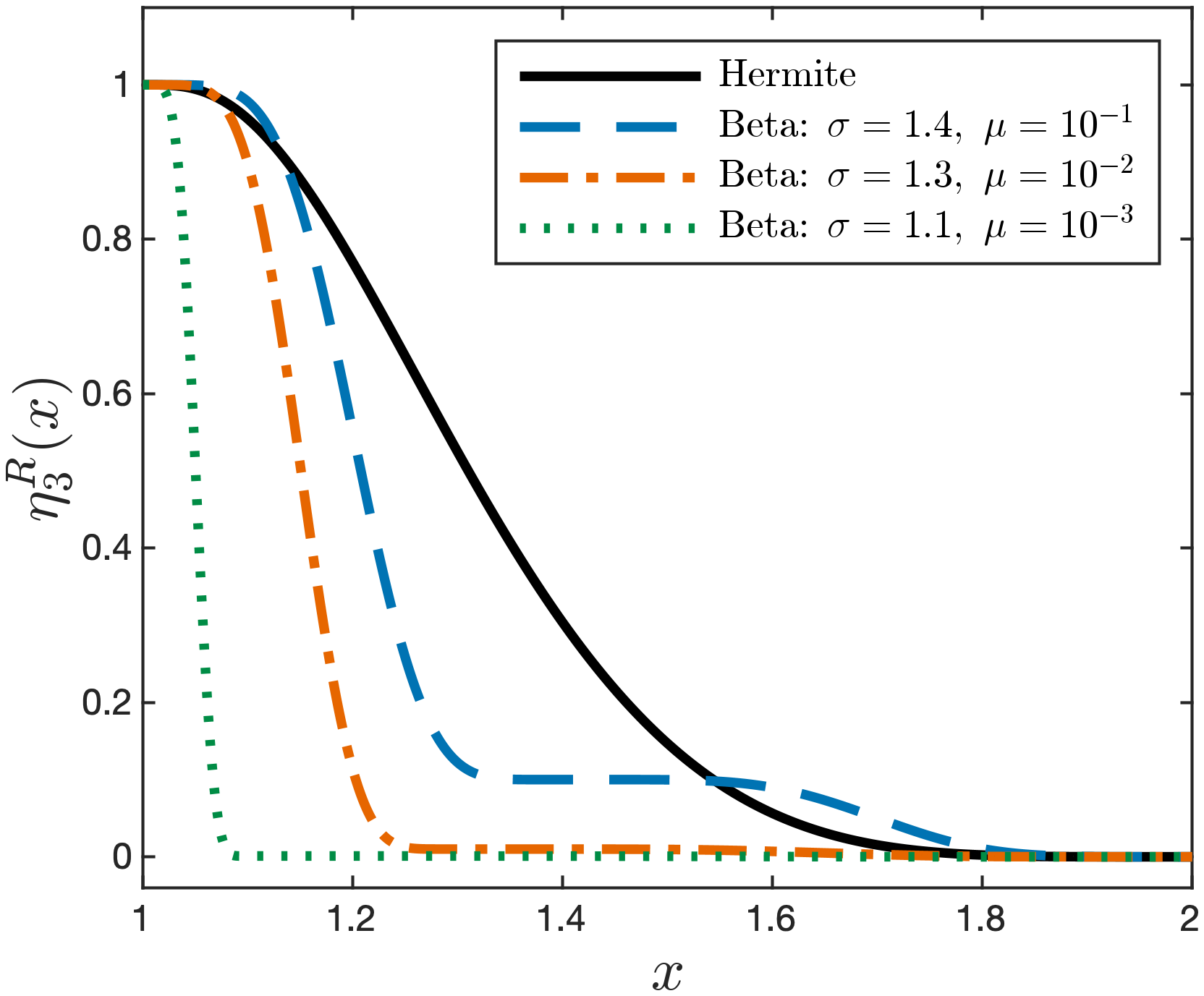}}%
\hfill
\subfloat{\includegraphics[width=3.2in]{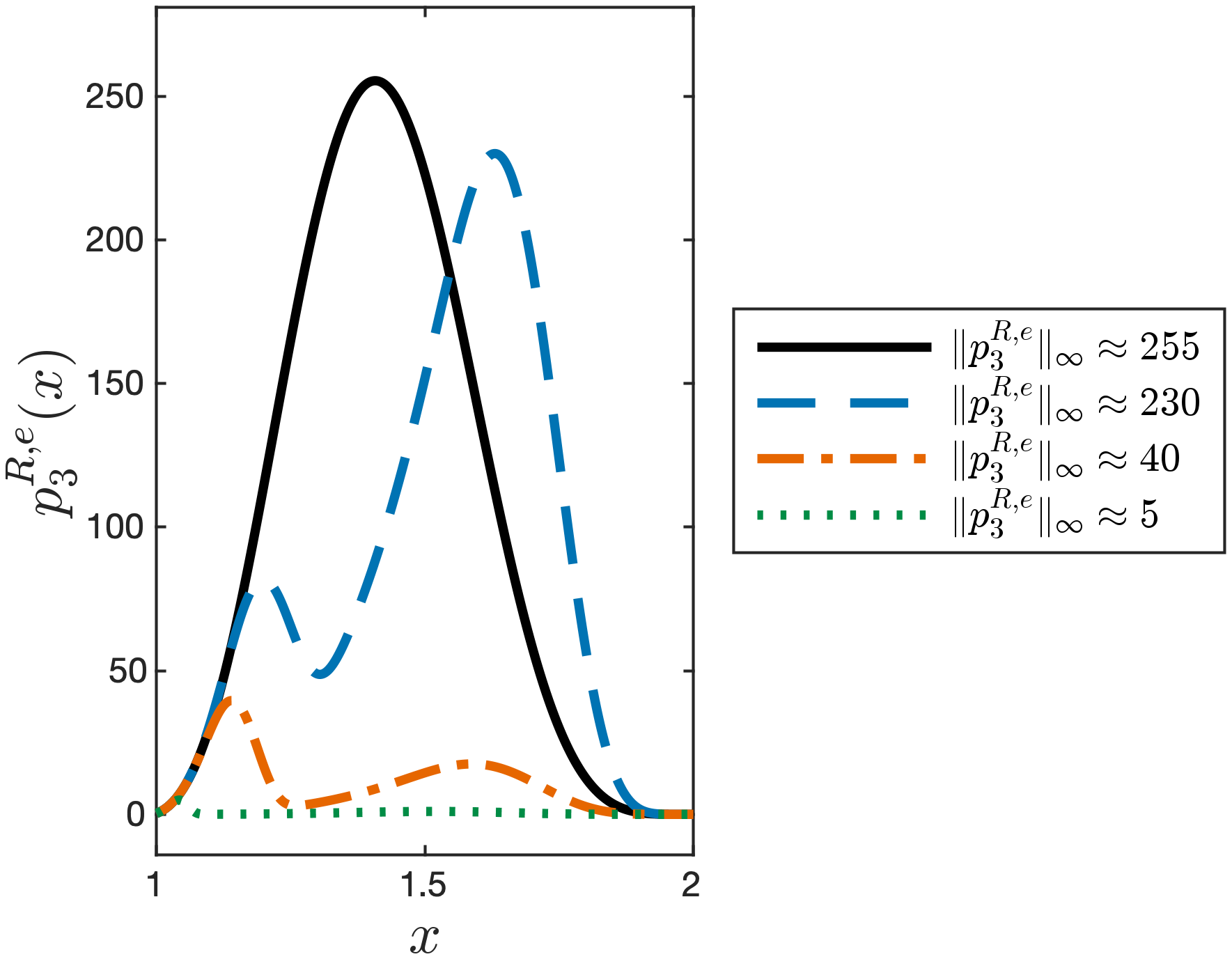}}
\caption{Left: Shape functions on $[1,\,2]$ for the regularized Beta family with several choices of $(\sigma_\ell,\, \mu_\ell)$ alongside the two-point Hermite polynomial of ModFC. Right: The corresponding blended continuations $p_3^{R,e}$, illustrating how the choice of shape parameters controls the supremum norm of the continuation.}
\label{fig:ModFC_GenFC_shapefunction}
\end{figure}

One natural direction for shape function design is a two-stage strategy, where the shape function first drops from $1$ to an intermediate level $\mu_\ell$ on $(1,\sigma_\ell)$ and then from $\mu_\ell$ to zero on $(\sigma_\ell, b)$. This additional degree of freedom provides finer control over the sup-norm of the blending continuation, as illustrated in Figure~\ref{fig:ModFC_GenFC_shapefunction} for the Beta shape function with various choices of $(\sigma_\ell, \mu_\ell)$. A rigorous study of optimal parameter selection within this family, a convergence analysis of the ODE and PDE solvers, and the extension of the framework to higher spatial dimensions are all directions we plan to pursue.

\bibliographystyle{plain}
\bibliography{references}

\end{document}